\documentclass[10pt, reqno]{amsart}
\usepackage{amsmath,amsfonts,amssymb,amsthm,xcolor}
\usepackage{epsfig}

\usepackage[normalem]{ulem}  
\usepackage{soul}

\usepackage{bbm}
\usepackage{enumerate}
\usepackage{amstext}
\usepackage{mathrsfs}
\usepackage{xspace}
\usepackage{amsfonts}
\usepackage{amsmath}
\usepackage{amssymb}
\usepackage{amstext}
\usepackage{amsthm}       %proof environment :)
\usepackage{xspace}
\usepackage[active]{srcltx}
\usepackage{cite}

\newcommand{\cyl}{(0,+\infty)\times\R^d}

\newcommand{\TUcyl}{(0,T)\times U}
\newcommand{\cTUcyl}{[0,T]\times\overline U}

\newcommand{\Tcyl}{(0,T)\times \R^d}
\newcommand{\cTcyl}{[0,T]\times\R^d}

\newcommand{\Qtau}{Q_\tau}
\newcommand{\QtauC}{\overline Q_\tau}

\newcommand{\D}[1]{\mbox{\rm #1}}

\newcommand{\I}{\mathcal I}

\newcommand{\R}{\mathbb R}

\newcommand{\N}{\mathbb N}

\newcommand{\Sol}{\mathcal S}

\newcommand{\Tr}{\D{tr}}
\newcommand{\Z}{\mathbb Z}

\newcommand{\eps}{\varepsilon}
\renewcommand{\epsilon}{\varepsilon}

\newcommand{\tagliato}{$\kern-5 mm -$}
\newcommand{\tagliat}{$\kern-4 mm -$}

\newcommand{\cchi}{\mbox{\large $\chi$}}

\newcommand{\ucv}{\rightrightarrows_{\text{\tiny loc}}}

\newcommand{\Ham}{\mathcal{H}\left(m,\Lambda,(a_r)_{r>0},(M_r)_{r>0}\right)}

\newcommand{\Haml}{\mathcal{H}\left(\ell,\Lambda,(a_r)_{r>0},(M_r)_{r>0}\right)}
\newcommand{\Hamtilde}{\mathcal{H}\big(m,\tilde\Lambda,(\tilde a_r)_{r>0},(\tilde M_r)_{r>0}\big)}

\newcommand{\Bam}{\mathfrak{B}\left(m,\Lambda,(a_r)_{r>0},(M_r)_{r>0}\right)}
\newcommand{\Bami}{\mathcal{B}\left(m_i,\Lambda,(a_r)_{r>0},(M_r)_{r>0}\right)}
\newcommand{\Bamtilde}{\mathcal{B}\big(m,\tilde\Lambda,(\tilde a_r)_{r>0},(\tilde M_r)_{r>0}\big)}
\newcommand{\Baml}{\mathcal{B}\left(\ell,\Lambda,(a_r)_{r>0},(M_r)_{r>0}\right)}
\newcommand{\Bamltilde}{\mathcal{B}\big(\ell,\tilde\Lambda,(\tilde a_r)_{r>0},(\tilde M_r)_{r>0}\big)}

\newtheorem{teorema}{Theorem}[section]
\newtheorem{prop}[teorema]{Proposition}
\newtheorem{lemma}[teorema]{Lemma}
\newtheorem{definition}[teorema]{Definition}

\newtheorem{guess}[teorema]{Remark}
\newtheorem{example}[teorema]{Example}

\newenvironment{esempio}{\begin{example} \begin{rm}}{\end{rm} \end{example}}

\begin{document}

\title{Existence and uniqueness of solutions to parabolic equations with superlinear Hamiltonians}
\author{Andrea Davini}
\address{Dip. di Matematica, {Sapienza} Universit\`a di Roma,
P.le Aldo Moro 2, 00185 Roma, Italy}

\email{davini@mat.uniroma1.it}


\begin{abstract}
We give a proof of existence and uniqueness of viscosity solutions to parabolic quasilinear equations for a fairly general class of nonconvex  Hamiltonians with superlinear growth in the gradient variable. The approach is mainly based on classical techniques for  uniformly parabolic quasilinear equations and on the Lipschitz estimates proved in \cite{AT}, as well as on  viscosity solution arguments. 
\end{abstract}
\date{Submitted Version February 6, 2017}
\maketitle
%
%22

\section*{Introduction}
In this paper we prove existence and uniqueness of viscosity solutions to a parabolic quasilinear equation of the form 
\begin{equation}\label{intro eq parabolic}
{\partial_t u}-\D{tr}(A(x)D_x^2u)+H(x, D_x u)=0\quad \hbox{in $ (0,T)\times \R^d$}
\end{equation}
subject to bounded uniformly continuous initial data.  Here $A$ is a $d\times d$ symmetric and positive semi-definite matrix
with Lipschitz and bounded coefficients, and the Hamiltonian $H$ is a
locally Lipschitz function on $\R^{d}\times\R^d$, which has superlinear growth in the gradient
variable but is not necessarily convex. The precise conditions we assume on $H$ will be discussed later. 
Our interest for this issue originates from our recent work \cite{DK16}, where this type of results are needed 
for the study of related homogenization problems.

Existence and uniqueness results for equations of this kind are usually derived either via the classical approach to quasilinear 
parabolic equations, or from suitable comparison principles for semicontinuous 
viscosity sub and supersolutions through a standard application of Perron's method. 

The classical parabolic theory yields existence and uniqueness of classical solutions 
provided the diffusion matrix $A$ is regular enough and uniformly positive definite, and the nonlinearity $H$ grows at most 
quadratically with respect to the gradient variable, see \cite[Chapter V, \S 8]{Lady}. 

The second approach is, on the other hand, more flexible, 
but the comparison results available in literature are usually proved under a uniform continuity condition on $H$ of the form 
\[
|H(x,p)-H(y,p)|\leqslant \omega\big((1+|p|)|x-y|\big)\qquad\hbox{for all $x,y,p\in\R^d$,}
\]
for some continuity modulus $\omega$, see for instance \cite[hypothesis (3.14)]{users}, \cite[hypothesis (H2)]{BBBL03},   \cite[hypothesis (H1)]{BBL03}.  
Such a condition is typically not satisfied by Hamiltonians with superlinear growth in $p$ as soon as the dependence in $x$ and $p$ is not decoupled.
The case of Hamiltonians with superlinear growth in $p$ of polynomial type has been specifically addressed in \cite{daLioLey06, daLioLey08} for a 
class of equations and of initial data that includes ours as a special instance. The Hamiltonians therein considered may also depend on $t$ and are not uniformly superlinear with respect to $x$, but unfortunately the techniques employed allow the authors to treat only the case of $H$ that is either convex in $p$, as in  \cite{daLioLey08},  or the sum of a convex and a concave one, where either one of the two grows at most linearly with respect to $p$, see \cite{daLioLey06} and \cite[Remark 2.1]{daLioLey08}. 

Several results holding for viscous Hamilton--Jacobi equations have been recently presented in \cite{AT} for a fairly general 
class of $t$--independent Hamiltonians with superlinear growth in $p$. The precise conditions assumed on $H$ are the hypotheses  
(H3) and (H4) with $\mu=+\infty$ listed in Section \ref{subsection viscosity solutions} below. 
Stationary Hamilton--Jacobi equations are also considered, but we will restrict our discussion here to the parabolic case.
%parabolic equations of the form \eqref{intro eq parabolic}. 
The authors prove two kind of results: comparison principle for semicontinuous sub and supersolutions of \eqref{intro eq parabolic} with, let us say, sublinear growth at infinity, see \cite[Theorem 2.3]{AT}; and interior Lipschitz estimates for continuous solutions of \eqref{intro eq parabolic} whose time--derivative satisfies  a uniform bound from below, see \cite[Proposition 3.5]{AT} or Proposition \ref{prop Lip estimates} in the next section. The comparison principle is proved by employing techniques close to the ones used in \cite{daLioLey08}. For this, it is crucial to additionally assume $H$  convex in $p$. On the contrary, the Lipschitz estimates are independent of this convexity condition, which is therefore dropped. Moreover, the authors provide a quantitative estimate of such Lipschitz constants in terms of the parameters that appear in the structural hypotheses (H3)--(H4) below. This is very convenient when one is, for instance, interested in approximating a given Hamiltonian in this class. 

The present work is aimed at removing the convexity condition on $H$ from the existence and uniqueness part of the  quoted results of \cite{AT}. 
The existence results are herein established under the regime of conditions (H3)--(H4), while the uniqueness is obtained by proving suitable comparison principle for semicontinuous sub and supersolutions to \eqref{intro eq parabolic} with sublinear growth at infinity. In the case 
of uniformly continuous Hamiltonians, i.e. when (H4) holds with constants $a_r$, $M_r$ independent of $r$, such a comparison principle follows rather easily from the existence part. In this instance, in fact, the solutions constructed in the first part are globally Lipschitz in $\cTcyl$ 
whenever the initial datum belongs to $\D{C}^\infty(\R^d)\cap W^{2,\infty}(\R^d)$ and in order to compare them with a semicontinuous sub or supersolution we just need a mild uniform continuity property on $H$, which holds true in view of conditions (H3) and \eqref{h1.2} in (H4), see Proposition \ref{prop Lip comp}. By exploiting the density of such initial data in the class of bounded uniformly continuous functions, a general comparison principle for semicontinuous sub and supersolutions is finally derived, see Theorem \ref{teo comp}. 

When the Hamiltonian is not uniformly superlinear, this idea can no longer be applied since the solutions are, in the best case scenario, only locally Lipschitz in $\cTcyl$. 
To deal with this case, we revisit the arguments employed in \cite[Section 2]{AT} and propose a minor 
generalization of \cite[Theorem 2.3]{AT} for Hamiltonians that satisfy (H3)--(H4) with $\mu=+\infty$ and that can be written as the pointwise infimum of a collection of convex Hamiltonians $\{H_i\}_{i\in\I}$ of same type, where the constants that appear in the structural conditions do not depend on the index $i$, see Theorem \ref{teo parabolic comp}. Actually, we allow the associated exponents $m$ to possibly depend on $i$, and we remark that we do not need to assume neither condition \eqref{h1.2} nor even continuity with respect to $x$ for such $H_i$. This can be useful for applications, see Example \ref{ex 2}. 

The existence part is the core of this work. Our approach mimic the classical one for uniformly parabolic quasilinear equations, based on the use of the Schauder fixed point Theorem and on suitable {\em a priori}  $L^\infty$ and  H\"older estimates on the gradient of the solutions, with the difference that, in order to have the necessary compactness to apply these tools, we  approximate \eqref{intro eq parabolic}  with a sequence of periodic parabolic equations of the same type with diverging size of  periodicity. The advantage  is that, in this way, we  just need {\em a priori interior} $L^\infty$ and  H\"older estimates on the gradient of the solutions for an equation of the form \eqref{intro eq parabolic}. 
For the former we directly apply \cite[Proposition 3.5]{AT}, while for the latter we use more classical results, see \cite[Chapter VI, Theorem 1.1]{Lady}. 
The fact that we have an explicit expression for such $L^\infty$ bounds is crucial for the remainder of the proof.  We stress that conditions (H3)--(H4) could be replaced by any other set of assumptions yielding similar $L^\infty$ bounds, but it is important to have an explicit expression for them in order to be able to control the local Lipschtiz constants of the approximating solutions that intervene in the limiting procedures we bring into play.  

The arguments  we employ are not new and are certainly known to some experts, see for instance \cite[Section 4]{BaSou2001} or \cite[Section 3]{LS_viscous}, however we could not 
locate in literature any reference where the issues herein considered have been proved in this generality, at least as far as the case of uniformly superlinear Hamiltonians is concerned. Our main motivation to write this note was to provide a reference for this kind of results.
We hope this work could be useful for other researchers working in this domain.\smallskip
\subsection*{Plan of the paper.}
Section \ref{sez preliminaries} contains some preliminary material. 
In Section \ref{subsection notation} we fix notation and 
define the functional spaces we use in the paper. In particular, we define the H\"older and parabolic H\"older spaces and their norms, 
and recall an interpolation inequality 
and a compact immersion result we will need for the existence part. Section \ref{subsection viscosity solutions} contains our standing assumptions 
on the diffusion matrix $A$ and on the Hamiltonian $H$ and some viscosity solution preliminaries. 
The existence results are derived in Section \ref{sez existence}. In Section \ref{subsection classical} we deal with the uniformly parabolic case, while in Section \ref{subsection general existence} we derive the existence result in the general case. The uniqueness part is treated in Section \ref{sez comparison}. In Section \ref{subsection uniformly superlinear} we deal with the uniformly superlinear case, while Section \ref{subsection nonuniformly superlinear} is devoted to the case of non--uniformly superlinear Hamiltonians. The proof of the comparison principle stated in Theorem \ref{teo parabolic comp} is postponed to the Appendix. In Section \ref{subsection examples} we give some examples of non--uniformly superlinear Hamiltonians covered by our study.
\medskip

\indent{\textsc{Acknowledgements. $-$}}
I am grateful to Panagiotis Souganidis for a brief but illuminating discussion we had at 
the conference {\em Hamilton-Jacobi Equations: new trends and applications} held at Rennes (France), 30 May-3 June, 2016. 
I would also like to thank Annalisa Cesaroni and Fabio Punzo for references and suggestions concerning the classical theory of parabolic equations, 
 Elena Kosygina for her useful comments on earlier versions of this paper, Scott Armstrong and Hung Tran for their prompt replies to my queries about their  
 joint work \cite{AT} and related issues. 
\bigskip

\section{Preliminaries}\label{sez preliminaries}

\numberwithin{equation}{section}

\subsection{Notation and functional spaces}\label{subsection notation}
Throughout the paper, we will denote by $\N$ and $\N_0$ the set of positive and nonnegative integer numbers, respectively. We will denote by  $\langle\cdot,\cdot\rangle$ and $|\cdot|$ the scalar product and the Euclidean norm on $\R^d$, where $d$ is a positive integer number. 
We will denote by $B_r(x_0)$ and $B_r$ the open balls in $\R^d$ of
radius $r$ centered at $x_0$ and $0$, respectively. 
For a given a subset $E$ of $\R^d$ or of $\R^{d+1}$, we will denote by $\overline E$ its closure.

Given a metric space $X$, we will write
$\varphi_n\ucv\varphi$ on $X$ to mean that the sequence of
functions $(\varphi_n)_n$ uniformly converges to $\varphi$ on
compact subsets of $X$.  We will denote by $\D{C}(X)$, $\D{UC}(X)$, $\D{LSC}(X)$, $\D{USC}(X)$ the space of continuous, uniformly continuous, lower semicontinuous, upper semicontinuous real functions on the metric space $X$, respectively. We will add the subscript $b$ to those spaces to mean that we are considering functions that are also bounded on $X$.  

Given an open subset $\Omega$ of either $\R^d$ or $\R^{d+1}$ and a measurable function $g:\Omega\to\R$, we will denote by $\|g\|_{L^\infty(\Omega)}$ its usual $L^\infty$--norms. We will denote $L^\infty(\Omega)$ the space of essentially bounded functions on $\Omega$, and by $W^{k,\infty}(\Omega)$ the space of functions $u\in L^\infty(\Omega)$ having essentially bounded distributional derivatives up to order $k\in\N$, inclusively.

Let $D$ be a smooth domain of $\R^d$ and $k\in\N$.  We will denote by $\D{C}^k(D)$ the space of continuous functions $u:D \to\R$ that are differentiable in $D$ with continuous derivatives up to order $k$ inclusively,  and by $\D{C}^\infty(D):=\bigcap_{k\in\N} \D{C}^k(D)$. 
% we will denote by $\D{C}^k_b(D)$ the subspace of $\D{C}^k(D)$ of bounded functions on $D$, and by $\D{C}^\infty_b(D):=\bigcap_{k\in\N} \D{C}^k_b(D)$.
We will denote by $\D{C}^k(\overline D )$ the space of continuous functions $u:\overline D \to\R$ that are differentiable in $D $ with continuous derivatives on $\overline D $ up to order $k$ inclusively.  In what follows, the letter $s$ refers to a multi--index, namely $s=(s_1,\dots,s_d)\in(\N_0)^d$, the symbol $|s|$ refers to the quantity $s_1+\dots+s_d$, and with the symbol $D^su$ or $D^s_x u$ we mean $\partial^{s_1}_{x_1}\dots\partial^{s_d}_{x_d}u$.\smallskip 

Let $k\in\N$ and $\alpha\in (0,1)$. 
For $u\in \D{C}^{k}(\overline D )$ we set 
\begin{eqnarray}\label{norm H}
\|u\|_{H^{k+\alpha}(D )}:=\sum_{|s|\leqslant k}\|D^s u\|_{L^\infty(D )}+\sum_{|s|=k}  [ D^s u ]^{(\alpha)}_D ,
\end{eqnarray}
with
\[
  [\varphi ]^{(\alpha)}_D :=\sup_{\substack{x,y\in D \\ x\not=y}}\frac{|\varphi(x)-\varphi(y)|}{|x-y|^\alpha}.
\]
%and $\|\varphi\|^{(0)}_D:=\|\varphi\|_{L^\infty(D)}$. 
We define 
\[
H^{k+\alpha}(\overline D ):=\{u\in \D{C}^{k}(\overline D )\,:\, \|u\|_{H^{k+\alpha}(D )}<+\infty\,\}. 
\]
The H\'older space  $H^{k+\alpha}(\overline D )$, endowed with the norm \eqref{norm H}, is a Banach space, see \cite{Lady}.

We record here for later use the following density result.

\begin{lemma}\label{lemma density}
The space of functions $\D{C}^\infty(\R^d)\cap W^{2,\infty}(\R^d)$ is dense in $\D{UC}_b(\R^d)$ with respect to the $\|\cdot\|_{L^\infty(\R^d)}$ norm.
\end{lemma}

\begin{proof}
Since $W^{1,\infty}(\R^d)$ is dense in $\D{UC}_b(\R^d)$ with respect to the $\|\cdot\|_{L^\infty(\R^d)}$ norm, see for instance \cite[Theorem 1]{GaJa}, it is enough to show that any Lipschitz and bounded function $g:\R^d\to\R$ can be uniformly approximated in $\R^d$ by functions in $\D{C}^\infty(\R^d)\cap W^{2,\infty}(\R^d)$. But this readily follows by regularizing $g$ via a convolution with a standard mollification kernel. 
\end{proof}

For a given $T>0$ and a smooth domain $D$ of $\R^d$, we will denote by $D_T$ the set $(0,T)\times D$. 
%we will denote by $\D{C}^{k/2,k}(D_T)$ the space of functions $u:D_T\to\R$ that are continuous in $D_T$ together with all derivatives of the form 
%$\partial_t^rD_x^s u$ for $2r+|s|\leqslant k$. 
We will denote by $\D{C}^{k/2,k}(\overline D_T)$ the space of functions $u:\overline D_T\to\R$ that are continuous 
in $\overline D_T$ together with all derivatives of the form $\partial_t^rD_x^s u$ for $2r+|s|\leqslant k$.

Let $\alpha\in (0,1)$. For $\psi\in\D{C}(\overline D_T)$, we set 
$[\psi]_{D_T}^{(\alpha)}:=  [ \psi ]_{t,D_T}^{(\alpha/2)}+  [ \psi ]_{x,D_T}^{(\alpha)}$, where 
\begin{align*}
  [ \psi ]_{t,D_T}^{(\alpha/2)}:=\sup_{x\in D} \|\psi(\cdot,x)\|_{H^{\alpha/2}((0,T))},\qquad
  [ \psi ]_{x,D_T}^{(\alpha)}:=\sup_{0<t<T} \| \psi(t,\cdot) \|_{H^\alpha(D)}.
\end{align*}
We introduce the following norms:
\begin{align*}
& \|u\|_{H^{\alpha/2,\alpha}(D_T)}:=\|u\|_{L^\infty(D_T)}+ [ u ]_{D_T}^{(\alpha)},\\
& \|u\|_{H^{(1+\alpha)/2,1+\alpha}(D_T)} :=\|u\|_{L^\infty(D_T)}+\sum_{i=1}^d\|\partial_{x_i} u\|_{H^{\alpha/2,\alpha}(D_T)}+ [ u ]_{t,D_T}^{\left(\frac{1+\alpha}{2}\right)},\\
&  \|u\|_{H^{(2+\alpha)/2,2+\alpha}(D_T)}:=
\| u\|_{L^\infty(D_T)}+ \sum_{i=1}^d\|\partial_{x_i} u\|_{H^{(1+\alpha)/2,1+\alpha}(D_T)}
+\| \partial_t u\|_{H^{\alpha/2,\alpha}(D_T)}.
\end{align*}
For $k\in\{0,1,2\}$, we define
\[
H^{(k+\alpha)/2,k+\alpha}(\overline D_T):=\{u\in \D{C}^{k/{2},k}(\overline D_T)\,:\, \|u\|_{H^{(k+\alpha)/2,k+\alpha}(D_T)}<+\infty\,\}. 
\]
The parabolic H\"older space $H^{k+\alpha/2,k+\alpha}(\overline D_T)$, endowed with the norm  
$\|\cdot\|_{H^{(k+\alpha)/2,k+\alpha}(D_T)}$, is a Banach space, see \cite{Lady}.

In the sequel we will often 
write 
\[
 \|D_x u\|:=\sum_{i=1}^d \|\partial_{x_i} u\|,\qquad \|D^2_x u\|:=\sum_{i,j=1}^d \|\partial^2_{x_ix_j} u\|, 
\]
where $u$ is a real function defined either on $D$ or on $D_T$ and $\|\cdot\|$ is a norm. 

We record the following result for further use:

\begin{prop}\label{prop parabolic interpolation}
Let $D$ be an open and convex subset of $\R^d$, $T>0$ and $\alpha\in (0,1)$. There exists a constant $N=N(d,D)$ such that for any $\eps>0$ and $u\in H^{(2+\alpha)/2,2+\alpha}(\overline D_T)$ we have 
\[
 \|D_x u\|_{H^{\alpha/2,\alpha}(D_T)}\leqslant 3\eps \|u\|_{H^{(2+\alpha)/2,2+\alpha}(D_T)}
 +{N}\max\{\eps^{-1/(1+\alpha)},\eps^{-(1+\alpha)} \}\|u\|_{L^\infty(D_T)}.
\]
\end{prop}

\begin{proof}
We apply  \cite[\S 8.8, Theorem 8.8.1]{Kr}. The assertion follows by summing the inequalities (8.8.3) and (8.8.4) and by noticing that 
\[
 [\partial_{x_i} u]_{\alpha/2,\alpha;D_T} 
 \geqslant
 \frac{1}{2}[\partial_{x_i} u]^{(\alpha)}_{D_T},
 \qquad
 [u]_{1+\alpha/2,2+\alpha;D_T}
 \leqslant
  \|u\|_{H^{(2+\alpha)/2,2+\alpha}(D_T)}.
\]
\end{proof}

For $n\in\N$, we will denote by $\D{C}_n^k(\R^d)$, $\D{C}_n^{k/2,k}(\cTcyl)$, $H_n^{(k+\alpha)/2,k+\alpha}(\cTcyl)$ the subspace of $\D{C}^k(\R^d)$, $\D{C}^{k/2,k}(\cTcyl)$, $H^{(k+\alpha)/2,k+\alpha}(\cTcyl)$, respectively, made up of functions that are $n\Z^d$--periodic in $\R^d$ with respect to the $x$--variable. We record for later use the following result, that can be easily proved with the aid of Ascoli--Arzel\`a Theorem.  

\begin{prop}\label{prop precompact}
Let  $n\in\N$, $\alpha\in (0,1)$ and $T>0$. The bounded subsets of the space $H_n^{(1+\alpha)/2,1+\alpha}(\cTcyl)$ are precompact in 
$H_n^{\alpha/2,\alpha}(\cTcyl)$.\medskip 
\end{prop}

\subsection{Viscosity solution theory}\label{subsection viscosity solutions}
In this paper we will consider parabolic quasilinear equations of the form 
\begin{equation}\label{eq parabolic}
{\partial_t u}-\D{tr}(A(x)D_x^2u)+H(x, D_x u)=0\quad \hbox{in $ (0,T)\times U$,}
\end{equation}
where $T>0$ and $U$ is an open subset of $\R^d$. The diffusion matrix $A(x)$ is a positive semidefinite symmetric $d\times d$ matrix, depending on $x\in\R^d$, with bounded and Lipschitz square root, namely $A=\sigma\sigma^T$ for some $\sigma:\R^d\to\R^{d\times n}$, where $\sigma$ satisfies the following hypotheses fos some fixed constant $\Lambda_A>0$:
\begin{itemize}
\item[(A1)] \quad $|\sigma(x)|\leqslant \Lambda_A$ \quad for every $x\in\R^d$;\smallskip
\item[(A2)] \quad $|\sigma(x)-\sigma(y)|\leqslant \Lambda_A |x-y|$ \quad for every $x,y\in\R^d$.\medskip
\end{itemize}
We emphasize that the diffusion matrix can be degenerate, in general. 

The nonlinearity $H$, henceforth called {\em Hamiltonian}, is a function $H:\R^d\times\R^d\to \R$ satisfying the following basic assumptions:
\begin{itemize}
\item[(H1)] there exist a continuous, coercive and nondecreasing functions $\Theta:\R_{+}\to\R$ and a constant $\mu\in\R$ such that 
\[
-\mu\leqslant H(x,p)\leqslant\Theta(|p|)\qquad\hbox{for every $(x,p)\in\R^d\times\R^d$};\smallskip
\]
\item[(H2)] \quad $H\in\D{UC}(\R^d\times B_r)$\quad for every $r>0$.\medskip
\end{itemize}
By coercive, we mean that $\displaystyle\lim_{h\to+\infty}{\Theta(h)}=+\infty$. The second inequality in (H1) amounts to saying that the Hamiltonian is locally bounded in $p$, uniformly with respect to $x$.

In order to obtain Lipschitz estimates for solutions to \eqref{eq parabolic}, we introduce another set of assumptions on $H$, holding for  constants $m>1$ and $\mu>0$:
\begin{itemize}
\item[(H3)] \quad $|H(x,p)-H(x,q)|\leqslant\Lambda\left(|p|+|q|+1\right)^{m-1}|p-q|$\quad for all $x,p,q\in\R^d$;\medskip
\item[(H4)] \quad for every $r>0$, there exist constants $a_r\in (0,1]$ and $M_r\geqslant 1$ such that 
\begin{align}
\max\left\{-\mu, a_r|p|^m-M_r \right\}\leqslant H(x,p)&\leqslant\Lambda(|p|^m+1) \label{h1.1}\\
|H(x,p)-H(y,p)|&\leqslant(\Lambda |p|^m+M_r)|x-y|  \label{h1.2}
\end{align}
\quad for all $x,y\in B_r$ and $p\in\R^d$.\medskip
\end{itemize}

When the above constants $\alpha_r,\,M_r$ can be chosen independently of $r$, we will say that the Hamiltonian is {\em uniformly superlinear}. 
Note that, in this instance, one can choose $\mu=+\infty$ in \eqref{h1.1}, as in \cite{AT}, and that condition (H2) is fulfilled. When on the other hand $H$ is not uniformly superlinear, condition (H2) needs not hold. 

Unless otherwise specified, all the differential inequalities in the paper are to be interpreted in the {\em viscosity} sense, which is the usual notion of weak solution for Hamilton--Jacobi equations. We briefly recall some basic definitions and refer  to \cite{barles_book, users} for further details. 

We will say that a function $v\in\D{USC}((0,T)\times U)$ is an (upper semicontinuous) {\em viscosity subsolution} of 
\eqref{eq parabolic} if, for every  $\phi\in\D{C}^2(\TUcyl)$ such that $v-\phi$ attains a local maximum at $(t_0,x_0)\in (0,+\infty)\times U$, we have 
\begin{equation*}\label{app subsolution test}
\partial_t \phi(t_0,x_0)-\Tr\big(A(x_0)D_x^2\phi(t_0,x_0)\big)+H(x_0, D_x \phi(t_0,x_0))\leqslant 0. 
\end{equation*}
Any such test function $\phi$ will be called {\em supertangent} to $v$ at $(t_0,x_0)$. 
%The second--order {\em superjet} $\Pcal^{2,+}v(t_0,x_0)$ of $v$ at $(t_0,x_0)$ is defined as the set of elements of the form $\left(\partial_t \phi(t_0,x_0),  D_x\phi(t_0,x_0),D^2_x\phi(t_0,x_0)\right)$ whenever $\phi$ is a $C^2$ supertangent to $v$ at $(t_0,x_0)$.
% \begin{align*}
% \Pcal^{2,+}v(t_0,x_0):= &\big\{\left(\partial_t \phi(t_0,x_0),  D_x\phi(t_0,x_0),D^2_x\phi(t_0,x_0)\right)\, :\\
% &\qquad\qquad\qquad\qquad\qquad \hbox{$\phi$ is a $C^2$ supertangent to $v$ at $(t_0,x_0)$}\,\big\} 
% \end{align*}

We will say that $w\in\D{LSC}((0,+\infty)\times U)$ is a (lower semicontinuous) {\em viscosity supersolution} of \eqref{eq parabolic} if, for every $\phi\in\D{C}^2(\TUcyl)$ such that $w-\phi$ attains a local minimum at $(t_0,x_0)\in (0,+\infty)\times \R^d$, we have 
\begin{equation*}\label{app supersolution test}
\partial_t \phi(t_0,x_0)-\Tr\big(A(x_0)D_x^2\phi(t_0,x_0)\big)+H(x_0, D_x \phi(t_0,x_0))\geqslant 0. 
\end{equation*}
Any such test function $\phi$ will be called {\em subtangent} to $w$ at $(t_0,x_0)$. 
%The definition of second--order {\em subjet} $\Pcal^{2,-}u(t_0,x_0)$ of $u$ at $(t_0,x_0)$ is analogous to that of superjet by replacing supertangents with subtangents. 
% 
% 
% The second--order {\em subjet} $\Pcal^{2,-}u(t_0,x_0)$ of $u$ at $(t_0,x_0)$ is defined as the set of elements of the form $\left(\partial_t \phi(t_0,x_0),  D_x\phi(t_0,x_0),D^2_x\phi(t_0,x_0)\right)$ whenever $\phi$ is a $C^2$ subtangent to $u$ at $(t_0,x_0)$.
It is well known, see for instance \cite{barles_book, users}, that the notion of sub or supertangent is local, in the sense that the test function $\phi$ needs to be defined only in a neighborhood of the point $(t_0,x_0)$. A continuous function on $\cyl$ is a {\em viscosity solution} of \eqref{eq parabolic} if it is both a viscosity sub and supersolution.\medskip

The following comparison principle holds:

\begin{prop}\label{prop Lip comp}
Assume that $A$ satisfy (A1)--(A2) and  $H\in\D{UC}\left( U\times B_r\right)$ for every $r>0$, where $U$ is an open subset of $\R^d$. Let $v\in\D{USC}([0,T]\times\overline U)$ and $w\in\D{LSC}([0,T]\times\overline U)$ be, respectively, a sub and a supersolution of \eqref{eq parabolic} satisfying 
\begin{equation}\label{hyp 2}
\limsup_{\substack{|x|\to +\infty\\ x\in U}}\ \sup_{t\in [0,T]}\frac{v(t,x)}{1+|x|}\leqslant 0 \leqslant \liminf_{\substack{|x|\to +\infty\\ x\in U}}\ \inf_{t\in [0,T]}\frac{w(t,x)}{1+|x|}.
\end{equation}
Let us furthermore assume that either $D_x v$ or $D_x w$ belongs to $\big(L^\infty\left((0,T)\times U\right)\big)^d$. Then \quad   
\[
v(t,x)-w(t,x)\leqslant \sup_{\partial_P\left(\TUcyl\right)}\big(v-w\big)\qquad\hbox{for every  $(t,x)\in (0,T)\times U$,} 
\]
where $\partial_P\left(\TUcyl\right):=\{0\}\times U\cup [0,T)\times\partial U$ is the parabolic boundary of $\TUcyl$. 
\end{prop}

The proof is standard, however we provide it in the Appendix for the reader's convenience. 

A first application of the above comparison principle is the following.

\begin{prop}\label{prop t-Lip}
Assume that $A$ satisfy (A1)--(A2) and  $H$ satisfies (H1)--(H2). Let $u\in\D{C}_b(\cTcyl)$ be a solution of \eqref{eq parabolic} with $U:=\R^d$ satisfying the initial condition $u(0,\cdot)=g$ for some $g\in\D{W}^{2,\infty}(\R^d)$. Let us furthermore assume that $D_x u\in\big(L^\infty(\Tcyl)\big)^d$. Then there exists a constant $\kappa$, only depending on $\|Dg\|_{L^\infty(\R^d)}$, $\|D^2 g\|_{L^\infty(\R^d)}$, $\mu$, $\Lambda_A$ and on the function $\Theta$, such that 
\[
 |u(t,x)-u(s,x)|\leqslant\kappa|t-s|\qquad\hbox{for all $(t,x),\,(s,x)\in \cTcyl$.}
\]
\end{prop}
\begin{proof}
Take a constant $\kappa$ large enough so that  
\[
\kappa> d\,\Lambda^2_A \|D^2 g\|_{L^\infty(\R^d)} + \max\left\{\mu, \Theta\left(\|D g\|_{L^\infty(\R^d)}\right)   \right\}.
\]
Then the functions $u_-(t,x):=g(x)-\kappa t$ and $u_+(t,x):=g(x)+\kappa t$ are, respectively, a bounded Lipschitz continuous sub and supersolution of \eqref{eq parabolic} with $U:=\R^d$. By Proposition \ref{prop Lip comp}, we infer that $u_-(t,x)\leqslant u(t,x)\leqslant u_+(t,x)$ for every $(t,x)\in\cTcyl$. For any fixed $h\in (0,T)$, the function $v(t,x):=u(t+h,x)$ is a bounded continuous solution to \eqref{eq parabolic} in $(0,T-h)\times\R^d$ with initial datum $v(0,\cdot)=u(h,\cdot)$. Furthermore, it is Lipschitz in $\Tcyl$ with respect to $x$, so by Proposition \ref{prop Lip comp} we infer
\[
\|u(t+h,\cdot)-u(t,\cdot)\|_{L^\infty(\R^d)}
\leqslant
\|u(h,\cdot)-u(0,\cdot)\|_{L^\infty(\R^d)}
\leqslant
\kappa\,h,
\]
yielding the claimed Lipschitz continuity of $u$ in $t$. 
\end{proof}

We recall the following crucial Lipschitz estimates for solutions to \eqref{eq parabolic} proved in \cite{AT}. 

\begin{prop}\label{prop Lip estimates}
Assume that $A$ satisfy (A1)--(A2) and  $H$ satisfies (H3)--(H4) with $\mu=+\infty$. Let $u\in\D{C}([0,T]\times\overline B_{r+1})$ be a solution of \eqref{eq parabolic} with $U:=B_{r+1}$ for some $r>0$, satisfying $u(0,\cdot)=g\in {W^{1,\infty}}(\overline B_{r+1})$ and 
\begin{equation*}
 \partial_t u\geqslant -\kappa\qquad\hbox{in $(0,T)\times B_{r+1}$}
\end{equation*}
for some positive constant $\kappa>0$. Then 
\[
 |u(t,x)-u(t,y)|\leqslant K_r|x-y|\qquad\hbox{for all $(t,x),(t,y)\in  (0,T)\times B_{r}$,}
\]
with $K_r>0$ given by 
\begin{equation}\label{def Lip constant}
K_r:=C\left\{\left(\frac{(1+\Lambda_A)^{1/2}\Lambda}{a_{r+1}}\right)^{1/(m-1)}+\left(\frac{M_{r+1}+\kappa}{a_{r+1}}\right)^{1/m}\right\},
\end{equation}
where $C$ is a positive constant only depending on $d$ and $m$.
\end{prop}

\section{Existence of  solutions}\label{sez existence}

The purpose of this section is to establish existence of solutions $u\in \D{C}_b(\cTcyl)$ to the equation
\begin{equation}\label{eq global parabolic}
{\partial_t u}-\D{tr}(A(x)D_x^2u)+H(x, D_x u)=0\quad \hbox{in $ (0,T)\times \R^d$,}
\end{equation}
subject to the initial condition $u(0,\cdot)=g\in\D{UC}_b(\R^d)$. We first deal with the uniformly parabolic case and show existence of  classical solutions 
to \eqref{eq global parabolic} when the initial datum is smooth enough, and then proceed to show the result in full generality.

\subsection{The uniformly parabolic case: existence of classical solutions}\label{subsection classical}
In this subsection we will show the existence of a solution $u\in \D{C}^{1,2}(\Tcyl)\cap \D{C}_b(\cTcyl)$ to \eqref{eq global parabolic}
subject to the initial condition $u(0,\cdot)=g\in\D{C}^\infty(\R^d)\cap W^{2,\infty}(\R^d)$ when the diffusion matrix is regular and uniformly positive definite. More precisely, throughout this subsection we will assume, besides (A1)--(A2), the following further assumptions on $A$: 
\begin{itemize}
\item[(A3)]\quad $A\in\D{C}^{1}(\R^d)$;\smallskip
\item[(A4)]\quad there exists a constant $\lambda>0$ such that 
\[
\langle A(x)\xi,\xi\rangle\geqslant \lambda |\xi|^2\qquad\hbox{for every $x,\xi\in\R^d$.}
\]
\end{itemize}
For the Hamiltonian, we will assume conditions (H3)--(H4).

The strategy we are going to implement is the following: we will approximate $A$, $H$ and $g$ with a sequence of diffusion matrices $A_n$, of Hamiltonians $H_n$ and of initial data $g_n$, that are $n\Z^d$--periodic in the $x$--variable and coincide with $A,\,H,\,g$, respectively, for $x$ belonging to a ball of radius $n/2$. The gain in compactness obtained in this way allows us to prove the existence of classical solutions $u_n$ for the approximating Cauchy problems. This is essentially achieved by following the classical approach to parabolic quasilinear equations, based on the use of  Schauder fixed point theorem and on suitable {\em a priori}  $L^\infty$ and H\"older estimates on the gradient of the solutions, see Proposition \ref{prop a priori estimates}. For the $L^\infty$ estimate, we will exploit Proposition \ref{prop Lip estimates}, while the H\"older estimates follow from more classical results. Then we will send $n\to +\infty$: since the functions $(u_n)_n$ are 
equi--bounded and 
locally equi--Lipschitz 
in $\cTcyl$, 
Ascoli--Arzel\`a Theorem, together with the stability properties of the notion of viscosity solution, implies that any accumulation point $u$ of the $(u_n)_n$ is a locally Lipschitz solution of \eqref{eq global parabolic} satisfying the initial condition $u(0,\cdot)=g$ on $\R^d$. The classical parabolic regularity theory (and Proposition \ref{prop Lip comp}) finally yields that such a $u$ is in $\D{C}^{1,2}(\Tcyl)$, hence a classical solution to \eqref{eq global parabolic}.\smallskip  

We proceed to implement the strategy outlined above. To this aim, choose $\cchi\in\D C^\infty(\R^d)$ so that $0\leqslant \cchi \leqslant 1$, $\cchi\equiv 1$ on $B_{1/2}$ and $\cchi\equiv 0$ on $\R^d\setminus B_{3/4}$. For every  $n\in\N$, we set 
\begin{align}%\label{def approximating}
&g_n(x):=g(x)\cchi(x/n)\qquad\qquad\qquad\qquad\qquad\qquad\qquad\qquad\hbox{for $x\in [-n,n]^d$,}\label{def gn}\\
&A_n(x):=A(x)\,\cchi(x/n)+\D{Id}\,(1-\cchi(x/n))\qquad\qquad\qquad\quad\ \hbox{for $x\in [-n,n]^d$,}\nonumber\\
&H_n(x,p):=H(x,p)\,\cchi(x/n)+\Lambda \big(|p|^m+1\big)\,(1-\cchi(x/n))\quad\hbox{for $(x,p)\in [-n,n]^d\times\R^d$,}\nonumber
\end{align}
and we extend them by periodicity to $\R^d$ and $\R^d\times\R^d$, respectively. 
Note that 
\begin{align}
g_n= g,\ A_n= A\quad\hbox{in $B_{n/2}$},\qquad H_n= H\quad&\hbox{in $B_{n/2}\times\R^d$}\label{property2 gn}
\end{align}
It is easily seen that the Hamiltonians $H_n$ satisfy (H3)--(H4), where the constants $a_r,\,M_r$ and $\Lambda$ can be chosen independent of $n$. Also note that, by periodicity, each $H_n$ is uniformly superlinear, i.e. (H4) holds with $\alpha_r=\alpha_n,\,M_r=M_n$ for every $r>0$. 
For each $n\in\N$, we define the quasilinear parabolic operator 
\[
P_n u:=\partial_t u - \D{tr}(A(x)D_x^2u)+H_n(x, D_x u)
\]
%
%In what follows, we will use the notation $E_\tau:=(0,\tau)\times B_n$ with $\tau>0$ and $n\in\N$. we will denote by $\partial_P E_\tau:=\{0\}\times B_n\bigcup \partial B_n\times [0,\tau)$ the parabolic boundary of $E_\tau$.\smallskip 

We start by deriving the {\em a priori}  Lipschitz and H\"older estimates.  

\begin{prop}\label{prop a priori estimates}
Suppose $u\in \D{C}_n^{1,2}([0,\tau]\times\R^d)$ satisfies $P_nu=0$ in $(0,\tau)\times\R^d$, $u(0,\cdot)=\phi$ on $\R^d $ 
with $\phi\in \D{C}_n^{2}(\R^d)$. Let $L>\|D\phi\|_{L^\infty(\R^d)}+\|D^2\phi \|_{L^\infty(\R^d)}$. Then 
\begin{equation*}
\|\partial_t u\|_{L^\infty((0,\tau)\times\R^d)}\leqslant \kappa,\qquad \|D_x u\|_{L^\infty((0,\tau)\times\R^d)}\leqslant K_n,
\end{equation*}
where $\kappa$ is a constant only depending on $L$, $\mu$, $\Lambda_A$, $\Lambda$, $m$, and $K_r$ is the constant given by \eqref{def Lip constant} with $r:=n$. Moreover, there exist  constants $\tilde C$ and $\alpha\in (0,1)$, only depending on $K_n,\,\Lambda,\,\lambda,\,\Lambda_A$ and $L$ (and independent of $\tau>0$, in particular), such that \ $\sum\limits_{i}\,[\partial_{x_i} u]^{(\alpha)}_{(0,\tau)\times\R^d} \leqslant \tilde C$. In particular, 
\begin{align}\label{Holder estimate}
\|u\|_{H^{(1+\alpha)/{2},1+\alpha}((0,\tau)\times\R^d)}\leqslant \kappa\tau+\|\phi\|_{L^\infty(\R^d)}+K_n+\tilde C+\kappa\tau^{\frac{1-\alpha}{2}}. 
\end{align} 
\end{prop}

\begin{proof}
By periodicity, the solution $u$ is clearly Lipschitz continuous in $(0,\tau)\times\R^d$. The Lipschitz estimates follow at once by  Propositions \ref{prop t-Lip} and \ref{prop Lip estimates}. The H\"older estimates on $D_x u$ can be derived by applying \cite[Chapter VI, Theorem 1.1]{Lady} with $\Omega:=[-2n,2n]^d$, $\Omega':=[-n,n]^d$. The inequality \eqref{Holder estimate} is a trivial consequence of these estimates. 
\end{proof}

We proceed by showing existence of a classical solution for the approximating parabolic Cauchy problems.

\begin{prop}\label{prop global existence}
There exists a function $u\in \D{C}_n^{1,2}(\cTcyl)$ that solves the problem
\begin{equation}\label{periodic parabolic problem}
 P_n u=0\quad\hbox{in $\cTcyl$,}\qquad u(0,\cdot)=g_n\quad\hbox{on $\R^d$.}
\end{equation}
\end{prop}

\begin{proof}
The proof is divided in two steps: we will first prove the local existence, i.e. the existence of a classical solution to \eqref{periodic parabolic problem} in $[0,\tau]\times\R^d$ for some $\tau\in (0,T]$; then we will prove that the maximal $\tau$ for which such a solution exists is equal to $T$. For notational brevity, throughout the proof we will write $\Qtau$ in place of $(0,\tau)\times\R^d$.\smallskip

\noindent{\em Step 1:} let $\tau\in (0,T]$ to be chosen and denote by $\alpha\in (0,1)$ the exponent provided by Proposition 
\ref{prop a priori estimates} with $g_n$ in place of $\phi$ and by $C$ the corresponding constant appearing at the right hand--side of \eqref{Holder estimate}. Set
\[
 \Sol:=\left\{v\in H_n^{{(1+\alpha)}/{2},1+\alpha}(\QtauC)\,:\,\|v\|_{H^{(1+\alpha)/{2},1+\alpha}(\Qtau)}\leqslant 2C\,    \right\}.
\]
Then we define a map $J:\Sol\to H_n^{{(1+\alpha)}/{2},1+\alpha}(\QtauC)$ by $u=Jv$, where $u$ solves the Cauchy problem 
\begin{equation*}\label{eq1 fixed point}
\partial_t u -\D{tr}(A_n(x)D_x^2u)+H_n(x, D_x v)=0\ \ \hbox{in $\QtauC$},\quad u(0,\cdot)=g_n\ \ \hbox{on $\R^d$.} 
\end{equation*}
Note that, by \cite[Chapter IV, Theorem 5.1]{Lady}, for each $v\in\Sol$, this problem has a unique solution $u\in H_n^{{(2+\alpha)}/{2},2+\alpha}(\QtauC)$ satisfying 
\begin{equation}\label{eq2 fixed point}
\|u\|_{H^{(2+\alpha)/{2},1+\alpha}(\Qtau)}\leqslant c\left(\|g_n\|_{H^{2+\alpha}(\Qtau)}+\|H_n(x,D_x v)\|_{H^{\alpha/2,\alpha}(\Qtau)}\right),
\end{equation}
where $c$ is a constant independent of $\tau\in (0,T]$, $g_n$ and $v$. Therefore 
\begin{equation}\label{eq3 fixed point}
\|Jv\|_{H^{(2+\alpha)/{2},2+\alpha}(\Qtau)}\leqslant \Gamma(C)\qquad\hbox{for all $v\in\Sol$,}
\end{equation}
for some continuous nondecreasing function $\Gamma:\R_+\to\R_+$, only depending on $H_n$. We proceed to show that we can choose $\tau\in (0,T]$ small enough so that \ $\|Jv\|_{H^{(1+\alpha)/{2},1+\alpha}(\Qtau)}\leqslant 2C$ \ for all $v\in\Sol$. To this aim, first note that, for $u=Jv$, we have 
\[
\|u\|_{H^{(1+\alpha)/{2},1+\alpha}(\Qtau)}
 \leqslant
 C
 +\|u-g_n\|_{H^{(1+\alpha)/{2},1+\alpha}(\Qtau)}. 
\]
Set $\tilde u:=u-g_n$ and let us estimate the term 
\[
 \|\tilde u\|_{H^{(1+\alpha)/{2},1+\alpha}(\Qtau)}=
 \|\tilde u\|_{L^\infty(\Qtau)}
 +
 \|D_x\tilde u\|_{H^{\alpha/2,\alpha}(\Qtau)}
 +
 [\tilde u]_{t,\Qtau}^{\left(\frac{1+\alpha}{2}\right)}.
\]
From the fact that $\tilde u(0,\cdot)=0$ on $\R^d$ and $\|\partial_t \tilde u\|_{L^\infty(\Qtau)}\leqslant \Gamma(C)$, we get 
\begin{equation*}
\|\tilde u\|_{L^\infty(\Qtau)}\leqslant \Gamma(C)\tau,
\qquad\qquad 
[\tilde u]_{t,\Qtau}^{\left(\frac{1+\alpha}{2}\right)}\leqslant \Gamma(C)\tau^{\frac{1-\alpha}{2}},
\end{equation*}
while the term $\|D_x\tilde u\|_{H^{\alpha/2,\alpha}(\Qtau)}$ can be controlled with the aid of Proposition \ref{prop parabolic interpolation}. We conclude that we can choose $\eps>0$ and a sufficiently small $\tau\in (0,T]$ so that 
\[
\|D_x\tilde u\|_{H^{\alpha/2,\alpha}(\Qtau)}\leqslant \frac{C}{2},
\qquad 
\|\tilde u\|_{L^\infty(\Qtau)}
+
[\tilde u]_{t,\Qtau}^{\left(\frac{1+\alpha}{2}\right)}
\leqslant
\frac{C}{2}.
\]
In particular, $J$ maps $\Sol$ into itself, for such a $\tau$. Since $\Sol$ is a convex and compact subset of the Banach space $H_n^{\alpha/{2},\alpha}(\QtauC)$, see Proposition \ref{prop precompact}, we can apply the Schauder fixed point Theorem, see for instance  \cite[Theorem 8.1]{Lieb}, and derive the existence of a fixed point $u$ of $J$, which is clearly in  $H_n^{{(2+\alpha)}/{2},2+\alpha}(\QtauC)$ and hence solves the Cauchy problem \eqref{periodic parabolic problem} in $[0,\tau]\times\R^d$.\smallskip

\noindent{\em Step 2:} let us set 
\[
T^*:=\sup\{\tau\in (0,T]\,:\,\hbox{\eqref{periodic parabolic problem} admits a solution in $\D{C}_n^{1,2}(\QtauC)$}\,\}.
\]
By the step 1, we know that the above set is nonempty. We want to show that $T^*=T$. To this aim, take a sequence $\big((\tau_k, u_k)\big)_{k}$ in $(0,T^*)\times\D{C}_n^{1,2}(\overline Q_{\tau_k})$, where $(\tau_k)_k$ converges increasingly to  
$T^*$ and $u_k$ solves \eqref{periodic parabolic problem} in $\overline Q_{\tau_k}$. 
From Proposition \ref{prop a priori estimates} we derive that each $u_k$ belongs to $H_n^{(1+\alpha)/{2},1+\alpha}(\overline Q_{\tau_k})$ and that there exist a constant $C>0$ and an exponent $\alpha\in (0,1)$, independent of $k\in\N$, such that  \ $\|u_k\|_{H^{(1+\alpha)/{2},1+\alpha}(Q_{\tau_k})}\leqslant C$ \ \ for every $k\in\N$. By applying \cite[Chapter IV, Theorem 5.1]{Lady} with $f(t,x):=H_n(x,D_xu_k)$ we infer that $u_k$ satisfies \eqref{eq2 fixed point} with $Q_{\tau_k}$ in place of $\Qtau$ and $u_k$ in place of $v$, where $c$ is a constant independent of $k$. We derive  
\begin{equation}\label{eq a priori estimate}
\|u_k\|_{H^{(2+\alpha)/{2},2+\alpha}(Q_{\tau_k})}\leqslant \Gamma(C)\qquad\hbox{for every $k\in\N$.} 
\end{equation}
Also notice that, by Proposition \ref{prop Lip comp},  
\begin{equation}\label{eq well posed}
u_k=u_h\quad\hbox{on $\overline Q_{\tau_h}$\ \ for every $k\geqslant h$.}
\end{equation}
We define a function $u:[0,T^*]\times\R^d\to\R$ by setting \ $u=u_k$\  on $\overline Q_{\tau_k}$\  for every $k\in\N$, and then by taking its continuous extension to $[0,T^*]\times\R^d$. According to \eqref{eq well posed} and \eqref{eq a priori estimate},  $u$ is well defined and belongs to ${H_n^{(2+\alpha)/{2},2+\alpha}(\overline Q_{T^*})}$. Moreover, it solves the Cauchy problem \eqref{periodic parabolic problem} in $[0,T^*)\times\R^d$ by construction, and also on $[0,T^*]\times\R^d$ by continuity of $u$, $\partial_t u$, $D_x u$, $D_x^2 u$. If, by contradiction, $T^*<T$, we could argue as in step 1 to find $\beta\in (0,1)$, $\tau\in (0,T-T^*)$ and a function $w\in H_n^{(1+\beta)/{2},1+\beta}\left([0,\tau]\times\R^d\right)$ such that 
\[
P_n w=0\quad\hbox{in $[0,\tau]\times\R^d$,}\qquad w(0,\cdot)=u(T^*,\cdot)\quad \hbox{on $\R^d$.} 
\]
It is easy to check that the function $u^*$ defined as 
\begin{eqnarray*}
u^*(t,x):=
\begin{cases}
u(t,x) & \hbox{if $(t,x)\in [0,T^*]\times\R^d$,}\\
w(t-T^*,x)& \hbox{if $(t,x)\in [T^*, T^*+\tau]\times\R^d$,}
\end{cases}
\end{eqnarray*}
belongs to $\D{C}_n^{1,2}(\overline Q_{T^*+\tau})$ and solves the Cauchy problem \eqref{periodic parabolic problem} in $\overline Q_{T^*+\tau}$, thus contradicting the maximality of $T^*$.
\end{proof}

We now proceed to prove the announced result.

\begin{teorema}\label{teo global parabolic}
Let $A$ satisfy (A1)--(A4) and let $H$ satisfy (H3)--(H4). Then, for every $g\in\D{C}^\infty(\R^d)\cap W^{2,\infty}(\R^d)$, there exists a classical solution $u\in \D{C}^{1,2}(\Tcyl)\cap \D{C}_b(\cTcyl)$ to \eqref{eq global parabolic} 
subject to the initial condition $u(0,\cdot)=g$ on $\R^d$. Moreover
\begin{equation}\label{claim Lip estimates}
\|\partial_t u\|_{L^\infty(\cTcyl)}\leqslant \kappa,\qquad \|D_x u\|_{L^\infty([0,T]\times B_r)}\leqslant K_r,
\end{equation}
where $\kappa$ is a constant only depending on $\|Dg\|_{L^\infty(\R^d)}$, $\|D^2 g\|_{L^\infty(\R^d)}$, $\mu$, $\Lambda_A$, $\Lambda$, $m$, and $K_r$ is the constant defined in \eqref{def Lip constant}.
\end{teorema}

\begin{proof}
In view of Proposition \ref{prop global existence}, for each $n\in\N$ there exists a solution 
$u_n\in \D{C}_n^{1,2}(\cTcyl)$  to the problem
\[
 P_n u=0\quad\hbox{in $\cTcyl$,}\qquad u(0,\cdot)=g_n\quad\hbox{on $\R^d$.}
\]
By combining Proposition \ref{prop t-Lip} and Proposition \ref{prop Lip estimates}, we get that the functions 
$u_n$ satisfy the Lipschitz estimates \eqref{claim Lip estimates}, at least eventually for every fixed $r>0$. 
By Ascoli--Arzel\`a Theorem and by possibly extracting a subsequence, we infer that there exists a function $u\in\D{C}(\cTcyl)$ such that $u_n\ucv u$ on $\cTcyl$. Since $H_n\ucv H$ on $\R^d\times\R^d$ and $A_n\ucv A$, $g_n\ucv g$ on $\R^d$, we infer that $u$ is a viscosity solution of \eqref{eq global parabolic} subject to the initial condition $u(0,\cdot)=g$ on $\R^d$. It is clear that $u$ satisfies \eqref{claim Lip estimates}. Being $g$ bounded on $\R^d$, we get in particular that $u$ in bounded on $\cTcyl$. 

Let us prove the asserted regularity of $u$. Let us fix $r>0$ and choose a smooth and bounded function $f_r:\R\to\R$ with $f_r(h)\equiv h$ on $[-z_h,z_h]$, with $z_h$ big enough so that \ $f_r(H(x,p))=H(x,p)$ for every $x\in B_r$ and $|p|\leqslant K_r$. Then $u$ is a viscosity solution of 
\[
{\partial_t u}-\D{tr}(A(x)D_x^2u)+f_r\big(H(x, D_x u)\big)=0\quad \hbox{in $ (0,T)\times B_r$.}
\]
On the other hand, \cite[Theorem 12.22]{Lieb} guarantees the existence of a solution $v\in\D{C}([0,T]\times\overline B_r)\cap\D{C}^{1,2}((0,T)\times B_r)$ satisfying the boundary condition $v=u$ on $\partial_P \big( (0,T)\times B_r\big)$. In view of the Comparison Principle stated in Proposition \ref{prop Lip comp} we infer that $u=v$ on $[0,T]\times\overline B_r$. The proof is complete.
\end{proof}
\smallskip

We end this subsection proving a comparison--type result for solutions to \eqref{eq global parabolic} obtained via approximation through periodic parabolic problems, as described above. 

\begin{prop}\label{prop comparison-type}
Let $A$ satisfy (A1)--(A4), $H$ satisfy (H3)--(H4) and $g^{1},\,g^{2}\in\D{C}^\infty(\R^d)\cap W^{2,\infty}(\R^d)$. Then there exists a pair $u^{1},\,u^2\in \D{C}^{1,2}(\Tcyl)\cap \D{C}_b(\cTcyl)$ of classical solutions to \eqref{eq global parabolic} subject to the initial condition $u^{i}(0,\cdot)=g^{i}$ on $\R^d$,  $i\in\{1,2\}$, satisfying 
\[
\|u^1-u^2\|_{L^\infty(\cTcyl)}\leqslant \|g^1-g^2\|_{L^\infty(\R^d)}.
\]
\end{prop}

\begin{proof}
For $i\in\{1,2\}$, let $g^i_n$ be the $n\Z^d$--periodic function on $\R^d$ defined via \eqref{def gn} and denote by $u^i_n\in \D{C}_n^{1,2}(\cTcyl)$  the solution to the problem
\[
 P_n u=0\quad\hbox{in $\cTcyl$,}\qquad u(0,\cdot)=g_n\quad\hbox{on $\R^d$}
\]
obtained according to Proposition \ref{prop global existence}. The functions $u^i_n$ are Lipschitz continuous on $\cTcyl$, hence, in view of Proposition \ref{prop Lip comp}, we infer 
\begin{equation}\label{eq comparison-type}
\|u_n^1-u_n^2\|_{L^\infty(\cTcyl)}\leqslant \|g^1_n-g^2_n\|_{L^\infty(\R^d)}\leqslant  \|g^1-g^2\|_{L^\infty(\R^d)}\quad\hbox{for each $n\in\N$.} 
\end{equation}
According to the proof of Theorem \ref{teo global parabolic}, there exists a pair $u^{1},\,u^2$ of bounded and continuous classical solutions to \eqref{eq global parabolic} subject to the initial condition $u^{i}(0,\cdot)=g^{i}$ on $\R^d$ such that, up to subsequences, \ $u_n^i\ucv u^i$\ in  $\cTcyl$ for $i\in\{1,2\}$.
The assertion follows by passing to the limit with respect to $n$ in \eqref{eq comparison-type}.
\end{proof}

\subsection{General existence results}\label{subsection general existence}

In this subsection we will prove existence of solutions to \eqref{eq global parabolic}, where we drop the regularity and uniform positivity conditions on the diffusion matrix, i.e. we will assume conditions (A1)--(A2) only.

\begin{teorema}\label{teo general existence}
Let $A$ satisfy (A1)--(A2) and let $H$ satisfy (H3)--(H4). Then, for every $g\in\D{UC}_b(\R^d)$, there exists a solution $u\in \D{C}_b(\cTcyl)$
to the equation \eqref{eq global parabolic} 
% \begin{equation}\label{claim main global parabolic}
% {\partial_t u}-\D{tr}(A(x)D_x^2u)+H(x, D_x u)=0\quad \hbox{in $ (0,T)\times \R^d$}
% \end{equation}
subject to the initial condition $u(0,\cdot)=g$ on $\R^d$. If $g\in\D{C}^\infty(\R^d)\cap W^{2,\infty}(\R^d)$, $u$ also satisfies
\begin{equation}\label{claim main Lip estimates}
\|\partial_t u\|_{L^\infty(\cTcyl)}\leqslant \kappa,\qquad \|D_x u\|_{L^\infty([0,T]\times B_r)}\leqslant K_r,
\end{equation}
where $\kappa$ is a constant only depending on $\|Dg\|_{L^\infty(\R^d)}$, $\|D^2 g\|_{L^\infty(\R^d)}$,  $\mu$, $\Lambda_A$, $\Lambda$, $m$, and $K_r$ is the constant defined in \eqref{def Lip constant}.
\end{teorema}

\begin{proof}
Let us first assume that $g\in\D{C}^\infty(\R^d)\cap W^{2,\infty}(\R^d)$. We introduce a sequence $(\rho_n)_n$  of standard mollifiers  and for each $n\in\N$ we set 
\[
\widetilde A_n(x):=\frac{1}{n^2}\D{Id}+(\rho_n* A)(x),\qquad x\in\R^d.
\]
In view of Theorem \ref{teo global parabolic}, for every $n\in\N$ there exists a classical solution $u_n\in \D{C}^{1,2}(\Tcyl)\cap \D{C}_b(\cTcyl)$ to the equation \eqref{eq global parabolic} with $\widetilde A_n$ in place of $A$ and subject to the initial condition $u_n(0,\cdot)=g$ on $\R^d$. Moreover this family of solutions satisfy \eqref{claim main Lip estimates} for some constants $\kappa$ and $K_r$ independent of $n$ (notice that $\Lambda_{\widetilde A_n}\leqslant \Lambda_{A}+1/n$). 
By Ascoli--Arzel\`a Theorem and by possibly extracting a subsequence, we infer that there exists a function $u\in\D{C}_b(\cTcyl)$ such that $u_n\ucv u$ on $\cTcyl$. Since $\widetilde A_n\ucv A$ on $\R^d$, we infer that $u$ is a viscosity solution of \eqref{eq global parabolic} subject to the initial condition $u(0,\cdot)=g$ on $\R^d$. It is clear that $u$ satisfies \eqref{claim main Lip estimates}.  

Let us now assume that $g\in\D{UC}_b(\R^d)$. Choose a sequence $(g^k)_k$ of initial data in $\D{C}^\infty(\R^d)\cap W^{2,\infty}(\R^d)$ such that \ $\|g-g^k\|_{L^\infty(\R^d)}\to 0$\  as $k\to +\infty$ and let us denote by $u^k$ a solution to \eqref{eq global parabolic} with initial datum $g^k$ obtained via the procedure described in the previous step. According to Proposition \ref{prop comparison-type} and by using a diagonal argument, this can be done in such a way that 
\[
\|u^k-u^h\|_{L^\infty(\cTcyl)}\leqslant \|g^k-g^h\|_{L^\infty(\R^d)}\qquad\hbox{for every $k, h\in\N$.}
\]
From the fact that $(g^k)_k$ is a converging sequence in $\D{C}_b(\R^d)$, we infer that $(u^k)_k$ is a Cauchy sequence in $\D{C}_b(\cTcyl)$. Therefore the solutions $u^k$ converge to some $u$ in $\D C_b(\cTcyl)$ and by stability we conclude that $u$ is a solution of \eqref{eq global parabolic} with initial datum $g$. 
\end{proof}

\section{Comparison Principles}\label{sez comparison}

In this section we are concerned with uniqueness properties of the solutions provided in the previous section, at least in the class of continuous bounded functions in cylinders of the form $\cTcyl$. This will be obtained as a consequence of the  comparison principles we will prove below.

\subsection{Uniformly superlinear Hamiltonians}\label{subsection uniformly superlinear}
In this subsection, we will deal with Hamiltonians satisfying (H3)--(H4) that are uniformly superlinear, i.e. for which (H4) holds with constants $a_r,\,M_r$ independent of $r>0$. 
In this case, the solutions to \eqref{eq global parabolic} with initial datum in $\D{C}^\infty(\R^d)\cap W^{2,\infty}(\R^d)$ provided by Theorem \ref{teo general existence} are globally Lipschitz in $\cTcyl$. Moreover, such Hamiltonians satisfy condition (H2) as well. We can therefore apply Proposition \ref{prop Lip comp} and, by exploiting the density of such initial data in $\D{UC}_b(\R^d)$, we can easily derive the following general Comparison Principle: 

\begin{teorema}\label{teo comp}
Assume $A$ satisfies (A1)--(A2) and  $H$ satisfies (H3)--(H4), with constants $a_r,\,M_r$ independent of $r>0$. Let $v\in\D{USC}([0,T]\times\R^d)$ and $w\in\D{LSC}([0,T]\times\R^d)$ be, respectively, a sub and a supersolution of \eqref{eq global parabolic} satisfying 
\begin{equation*}
\limsup_{|x|\to +\infty}\sup_{t\in [0,T]}\frac{v(t,x)}{1+|x|}\leqslant 0 \leqslant \liminf_{|x|\to +\infty}\inf_{t\in [0,T]}\frac{w(t,x)}{1+|x|}.
\end{equation*}
Let us furthermore assume that $v(0,\cdot)\leqslant g \leqslant w(0,\cdot)$ for some $g\in\D{UC}_b(\R^d)$. Then \quad   
\[
v(t,x)\leqslant w(t,x)\qquad\hbox{for every  $(t,x)\in\cTcyl$.} 
\]
\end{teorema}

\begin{proof}
Fix $\eps>0$ and set $w_\eps:=w+\eps$. Since 
$v(0,\cdot)\leqslant g<g+\eps\leqslant w_\eps(0,\cdot)$, in view of Lemma \ref{lemma density} we can find a function $g_\eps\in \D{C}^\infty(\R^d)\cap W^{2,\infty}(\R^d)$ such that $v(0,\cdot)\leqslant g_\eps \leqslant w_\eps(0,\cdot)$ on $\R^d$. By Theorem \ref{teo general existence} and by taking into account that $H$ is uniformly superlinear, there exists a solution $u_\eps\in \D{C}(\cTcyl)\cap W^{1,\infty}(\cTcyl)$ of \eqref{eq global parabolic} with initial datum $g_\eps$. Since $H$ satisfies (H2), we can apply the Comparison Principle stated in Proposition \ref{prop Lip comp} with $U:=\R^d$ to infer that \quad $v\leqslant u_\eps \leqslant w_\eps=w+\eps$ \ in $\cTcyl$. The assertion follows since $\eps>0$ was arbitrarily chosen.
\end{proof}

As a simple consequence of Theorems \ref{teo global parabolic} and \ref{teo comp} we derive the following result:

\begin{teorema}
Let $A$ satisfy (A1)--(A2) and let $H$ satisfy (H3)--(H4), with constants $a_r,\,M_r$ independent of $r>0$. Then, for every $g\in\D{UC}_b(\R^d)$, there exists a unique function $u\in \D{UC}_b(\cTcyl)$ that solves the equation \eqref{eq global parabolic}
% \begin{equation}\label{eq global parabolic}
% {\partial_t u}-\D{tr}(A(x)D_x^2u)+H(x, D_x u)=0\quad \hbox{in $ (0,T)\times \R^d$}
% \end{equation}
subject to the initial condition $u(0,\cdot)=g$ on $\R^d$. If $g\in W^{2,\infty}(\R^d)$, $u$ is Lipschitz continuous in $\cTcyl$ and satisfies
\begin{equation*}
\|\partial_t u\|_{L^\infty(\cTcyl)}\leqslant \kappa,\qquad \|D_x u\|_{L^\infty([0,T]\times \R^d)}\leqslant K,
\end{equation*}
where $\kappa$ is a constant only depending on $\|Dg\|_{L^\infty(\R^d)}$, $\|D^2 g\|_{L^\infty(\R^d)}$,  $\mu$, $\Lambda_A$, $\Lambda$, $m$, and $K$ is the constant, independent of $r>0$, defined in \eqref{def Lip constant}.
\end{teorema}

\begin{proof}
The uniqueness part is obvious in view of Theorem \ref{teo comp}. 
Let $g\in W^{2,\infty}(\R^d)$ and denote by $u$ the unique function in $\D{C}_b(\cTcyl)$ that solves \eqref{eq global parabolic}. The fact that $\|\partial_t u\|_{L^\infty(\cTcyl)}\leqslant \kappa$ for a constant $\kappa$ only depending on $\|Dg\|_{L^\infty(\R^d)}$, $\|D^2 g\|_{L^\infty(\R^d)}$,  $\mu$, $\Lambda_A$, $\Lambda$, $m$ is derived by arguing as in the proof of Proposition \ref{prop t-Lip} and by using Theorem \ref{teo comp} in place of Proposition \ref{prop Lip comp}. The $L^\infty$--bound on $D_xu$ follows by applying Proposition \ref{prop Lip estimates}.

Let us now assume that $g\in\D{UC}_b(\R^d)$. Choose a sequence $(g^k)_k$ of initial data in $W^{2,\infty}(\R^d)$ such that \ $\|g-g^k\|_{L^\infty(\R^d)}\to 0$\  as $k\to +\infty$ and denote by $u$, $u^k$ the unique solution to \eqref{eq global parabolic} in  $\D{C}_b(\cTcyl)$ with initial datum $g$, $g^k$, respectively. By Theorem \ref{teo comp} 
\[
\|u-u^k\|_{L^\infty(\cTcyl)}\leqslant \|g-g^k\|_{L^\infty(\R^d)}\qquad\hbox{for every $k\in\N$.}
\]
As a uniform limit of a sequence of Lipschitz functions, we conclude that $u\in \D{UC}_b(\cTcyl)$.
\end{proof}

\subsection{Non--uniformly superlinear Hamiltonians}\label{subsection nonuniformly superlinear} 
When the Hamiltonian $H$ is not uniformly superlinear, i.e. the constants $a_r, M_r$ in (H4) actually depend on $r>0$, Theorem \ref{teo general existence} provides us with solutions to \eqref{eq global parabolic} that are, in the best case scenario, only locally Lipschitz in $\cTcyl$ and the idea exploited in the previous subsection can no longer be used. We will therefore restrict our analysis to Hamiltonians of special form, by slightly relaxing the convexity condition in $p$ assumed in \cite{AT}. The results of this subsection are based on a technical refinement of the arguments therein employed.

It is convenient to introduce a piece of notation first. Let $m>1$, $\Lambda>0$, $(a_r)_{r>0}$ in $(0,1]$ and $(M_r)_{r>0}$ in $[1,+\infty)$ be fixed constants.
We will denote by $\Bam$ the family of Borel functions $F:\R^d\times\R^d\to\R$ that are convex in $p$ and satisfy (H3) and condition \eqref{h1.1} in (H4) with $\mu=+\infty$, and by $\Ham$ the family of Hamiltonians $H:\R^d\times\R^d\to\R$ satisfying conditions (H3) and (H4) with $\mu=+\infty$. Note that we are not 
assuming neither condition (H2) nor that $H$ is bounded from below.

We consider a Hamiltonian $H\in\Ham$ of the form
\begin{equation}\label{def H}
H(x,p)=\inf_{i\in\I} H_{i}(x,p),\qquad\hbox{for all $(x,p)\in\R^d\times\R^d$,}
\end{equation}
where $\I$ is a set of indexes and each $H_{i}$ belongs to $\Bami$, with exponent $m_i>1$ possibly depending on $i$. 
Notice that\ $m=\inf_{i} m_i$.

\begin{teorema}\label{teo parabolic comp}
Let $A$ satisfy (A1)--(A2) and $H$ as above. Let $U$ an open subset of $\R^d$ and let $v\in\D{USC}([0,T]\times\overline U)$ and $w\in\D{LSC}([0,T]\times\overline U)$ be, respectively, a sub and a supersolution of
\begin{equation}\label{eq parabolic comp}
{\partial_t u}-\D{tr}(A(x)D_x^2u)+H(x, D_x u)=0\quad \hbox{in $ (0,T)\times U$,}
\end{equation}
satisfying 
\begin{equation}\label{app hyp 2}
\limsup_{\substack{|x|\to +\infty\\ x\in U}}\ \sup_{t\in [0,T]}\frac{v(t,x)}{1+|x|}\leqslant 0 \leqslant \liminf_{\substack{|x|\to +\infty\\ x\in U}}\ \inf_{t\in [0,T]}\frac{w(t,x)}{1+|x|}.
\end{equation}
Then \quad   
\[
v(t,x)-w(t,x)\leqslant \sup_{\partial_P\left(\TUcyl\right)}\big(v-w\big)\qquad\hbox{for every  $(t,x)\in (0,T)\times U$,} 
\]
where $\partial_P\left(\TUcyl\right):=\{0\}\times U\cup [0,T)\times\partial U$ is the parabolic boundary of $\TUcyl$. 
\end{teorema}

For the proof of Theorem \ref{teo parabolic comp}, we will use in a crucial way the following estimate, that we will prove separately:

\begin{lemma}\label{lemma AT}
Let $H$ be as above. For fixed $\eta\in (0,1/8)$ and $R_\eta>1$, let $x_\eps,y_\eps,q_\eps\in\R^d$  such that $|x_\eps|,\,|y_\eps|\leqslant R_{\eta}-1$, $|q_\eps|\leqslant \eta$ for every $\eps\in (0,1)$, and
\begin{equation}\label{cond AT}
\lim_{\eps\to 0^+}{|x_\eps-y_\eps|}=0.
\end{equation}
Then there exist $\eps(\eta)>0$, $C>0$ and a constant $C_\eta>0$, depending on $\eta$, such that, for every $\eps<\eps(\eta)$ and for $s=1-4\eta$ we have
\begin{equation}\label{claim AT}
sH\left(x_\eps, \frac{p_\eps+q_\eps}{s}\right)-H\left(y_\eps, p_\eps\right)\geqslant -C(1-s)-C_\eta|x_\eps-y_\eps|,
\end{equation}
where $p_\eps:=(x_\eps-y_\eps)/\eps$.
\end{lemma}

\begin{proof}
The proof relies on the arguments used in \cite{AT}, up to some technical modifications that we detail below. 
Let us denote by $I$ the left--hand side term of \eqref{claim AT}. We have  
\begin{align*}\label{claim 1b equivalent}
I=\underbrace{ \left( sH\big(x_\eps, \frac{p_\eps+q_\eps}{s}\big)-H\left(x_\eps, p_\eps\right) \right)}_{I_1}
+
\underbrace{  \Big(H\left(x_\eps, p_\eps\right)-H\left(y_\eps, p_\eps\right) \Big)}_{I_2}.
\end{align*}
By the fact that $H$ satisfies hypothesis \eqref{h1.2} and $m\leqslant m_i$, we get 
\begin{equation}\label{estimate I2}
I_2\geqslant -\Lambda\left(|p_\eps|^m+M_{R_\eta}\right)|x_\eps-y_\eps|\geqslant  -\Lambda\left(|p_\eps|^{m_i}+2M_{R_\eta}\right)|x_\eps-y_\eps|
\quad
\hbox{for all $i\in\I$}.
\end{equation}
As for the term $I_1$, we obviously have $I_1\geqslant \inf_{i\in\I} J_i$ with  
\[
J_i:=sH_i\big(x_\eps, \frac{p_\eps+q_\eps}{s}\big)-H_i\left(x_\eps, p_\eps\right). 
\]
Let us estimate $J_i$, for each fixed $i$. 
Set $r:=(1+s)/2<1$. We exploit the convexity of $H_i$ in $p$: by arguing as in \cite{AT}, we get
\[
J_i\geqslant { \left( \frac{s}{r} H_i\left(x_\eps,\frac{r}{s}p_\eps\right)-H_i(x_\eps,p_\eps)\right)}-\frac{\Lambda}{2}(1-s)\left( 1+\frac{(4\eta)^{m_i}}{(1-s)^{m_i}} \right).
\]
Using the fact that $1-s=4\eta$, this inequality can be restated as 
\begin{equation}\label{ineq I_1} 
J_i\geqslant  \underbrace{ \left( \frac{s}{r} H_i\left(x_\eps,\frac{r}{s}p_\eps\right)-H_i(x_\eps,p_\eps)\right)}_{G_\eps}-\Lambda(1-s).
\end{equation}
We proceed to estimate $G_\eps$: by arguing as in \cite{AT} we get   
\[
G_\eps\geqslant \frac{1-s}{2}\left({\left(\gamma\, a_{R_\eta}- \Lambda\gamma^{m_i}\right)}|p_\eps|^{m_i}-\gamma M_{R_\eta}-\Lambda \right),
\] 
where $\gamma$ is any fixed parameter in $(0,1/2)$. Notice that, in view of the fact that $m_i\geqslant m$, we have 
$\gamma^{m_{i}}\leqslant \gamma^{m}$. We conclude that we can choose $\gamma$ sufficiently small in such a way that the term in front of $|p_\eps|^{m_i}$ can be estimated from below by $2C_\eta$, where $C_\eta$ is a positive constant only depending on $\eta$ (and, in particular, independent of $i$), and the term $\gamma M_{R_\eta}$ can be estimated from  above by $\Lambda$.  We get 
\begin{equation*}\label{ineq G}
J_i\geqslant G_\eps-\Lambda(1-s)\geqslant C_\eta(1-s)|p_\eps|^{m_i}-2\Lambda(1-s).
\end{equation*}
By taking into account \eqref{estimate I2}, we infer 
\begin{equation}\label{ineq final}
J_i+I_2\geqslant 
\big(C_\eta(1-s)-\Lambda|x_\eps-y_\eps| \big)|p_\eps|^{m_i}-2\Lambda(1-s)-2M_{R_\eta}|x_\eps-y_\eps|.
\end{equation}
Since $|x_\eps-y_\eps|\to 0$ as $\eps\to 0^+$ and $1-s=4\eta$, we infer that,  for every fixed $\eta>0$, we can find $\eps(\eta)>0$, independent of $i$, such that, for every $0<\eps<\eps(\eta)$, the term in front of $|p_\eps|^{m_i}$ is positive. By discarding it from the right--hand side of \eqref{ineq final} and by recalling that 
$I_1+I_2\geqslant \inf_{i\in\I} J_i+I_2$, we finally obtain \eqref{claim AT} with $C:=2\Lambda$ and $C_\eta:=2M_{R_\eta}$.
\end{proof}

With the aid of Lemma \ref{lemma AT}, the proof of Theorem \ref{teo parabolic comp} can be carried on by reasoning as in \cite{AT}. For the reader's convenience, we give it in the Appendix. We will furnish more details than in \cite{AT} and also correct  a misleading misprint therein contained, see \eqref{comp parabolic subsol}.\smallskip

We end this subsection by providing a slight generalization of Theorem \ref{teo parabolic comp}.

\begin{prop}\label{prop trivial}
Let $H\in\Ham$ and assume there exists $\rho>0$ such that 
\begin{equation}\label{def H bis}
H(x,p)=\inf_{i\in\I} H_{i}(x,p),\qquad\hbox{for all $(x,p)\in\R^d\times(\R^d\setminus B_\rho)$,}
\end{equation}
where $\I$ is a set of indexes and each $H_{i}$ belongs to $\Bami$, with exponent $m_i>1$ possibly depending on $i$.  
Then, for such a $H$, the statement of Theorem \ref{teo parabolic comp} holds. 
\end{prop}

\begin{proof}
Looking at the proof of Theorem \ref{teo parabolic comp}, it is clear that it suffices to prove that, for every fixed $\eta\in (0,1/8)$, there is an infinitesimal sequence $(\eps_k)_k$ such that $H$ satisfies \eqref{claim AT} in Lemma \ref{lemma AT} for every $\eps\in\{\eps_k\,:\,k\in\N\}$.  Therefore, let us fix $\eta\in (0,1/8)$, set $s:=1-4\eta$, and let $x_\eps,y_\eps,p_\eps, q_\eps$ as in the statement of that lemma. Then there exists an infinitesimal sequence $(\eps_k)_k$ such that either $|p_{\eps_k}|\leqslant \rho+1$ for all $k\in\N$, or 
$|p_{\eps_k}|> \rho+1$ for all $k\in\N$. Let $\eps=\eps_k$ with $k\in\N$. We follow the notation used in the proof of Lemma \ref{lemma AT}.  

\noindent In the first case, first notice that $|(p_\eps+q_\eps)/s|<2\rho+3$. From (H4) and (H3) we  
get 
\[
I_1
\geqslant 
-(1-s)\Lambda \big(1+(2\rho+3)^m\big)-\omega\left(\big|\frac{p_\eps+q_\eps}{s}-p_\eps\big|\right),
\]
where $\omega$ is a continuity modulus of $H(x_\eps,\cdot)$ in $B_{2\rho+3}$. In view of (H3) and of the relation $s=1-4\eta$, we infer that 
there exists a constant $C$, only depending on $m,\Lambda$ and $\rho$, such that 
\[
I_1\geqslant -C\,(1-s).
\]
As for $I_2$, from \eqref{estimate I2} we infer 
\[
 I_2\geqslant -\Lambda\left((\rho+1)^m+M_{R_\eta}\right)|x_\eps-y_\eps|.
\]

\noindent In the second case, notice that $|(p_\eps+q_\eps)/s|>\rho$. 
We set $\tilde H(x,p):=\max\{H(x,p),\mu(x)\}$ for every $(x,p)\in\R^d\times\R^d$, with $\mu(x):=\inf_{|p|\geqslant\rho} H(x,p)$. 
Now remark that such $\tilde H$ belongs to $\Hamtilde$ for suitable constants $\tilde\Lambda>0$, $(\tilde a_r)_{r>0}$ in $(0,1]$, $(\tilde M_r)_{r>0}$ in $[1,+\infty)$, and it  
can be written as in \eqref{def H} with $\max\{H_i(x,p),\mu(x)\}$ in place of $H_i$, for each $i\in\I$. We can therefore apply Lemma \ref{lemma AT} to $\tilde H$ and conclude that $H$ satisfies \eqref{claim AT} since $H=\tilde H$ on $\R^d\times\big(\R^d\setminus B_\rho\big)$, by definition of $\tilde H$. The proof is complete. 
\end{proof}

\subsection{Examples}\label{subsection examples}
In this subsection we give a couple of examples of Hamiltonians in the class $\Ham$ that can be written in the form \eqref{def H} for some functions $H_i\in\Bami$. 

\begin{esempio}\label{ex 1}
Let $H\in\Ham$ be of the form 
$$
H(x,p):=K(x,p)+G(x,p),\qquad (x,p)\in\R^d\times\R^d,
$$ 
where $K$ is a  convex Hamiltonian belonging to $\Ham$, while
\[
 G(x,p):=\inf_{i\in\I} \left\{\langle g_i(x),p\rangle +f_i(x)\right\},\qquad (x,p)\in\R^d\times\R^d,
\]
where the functions $g_i:\R^d\to\R^d$ and $f_i:\R^d\to\R$ are Borel measurable and equibounded. Then $H=\inf_{i\in\I} H_i$, where the functions  
$$H_i(x,p):=K(x,p)+\langle g_i(x),p\rangle +f_i(x)$$ 
belong to $\Bamtilde$, for suitable constants $\tilde\Lambda>0$, $(\tilde a_r)_{r>0}$ in $(0,1]$, $(\tilde M_r)_{r>0}$ in $[1,+\infty)$.
\end{esempio}

Our second example consists in considering $H\in\Ham$ such to satisfy a 
 semiconcavity--type condition in $p$ inside a compact set of momenta and a convexity condition in $p$  in the complement. This example is, of course, already covered by Proposition \ref{prop trivial}. We include it nevertheless to show why it is useful to drop continuity with respect to $x$ for the approximating functions $H_i$  and to allow the associated exponents to possibly depend on the index.
%
%As a matter of fact, this will be the object of a proposition. 
We remark that it is natural to expect some kind of semi--concavity property in $p$ for a Hamiltonian of the form \eqref{def H}. Indeed, the fact that the $H_i$ are convex in $p$ and are trapped between two paraboloids, according to condition \eqref{h1.1} in (H4), should entail, loosely speaking, a form of equi--semiconcavity in $p$ for the approximating functions $H_i$, locally with respect to $x$. 

\begin{esempio}\label{ex 2}
Let $H\in\Ham$ and assume there exist $\rho>0$ and $K\in\Baml$ for some $\ell\geqslant m$ such that 
\begin{itemize}
 \item[(a)] $H(x,\cdot)-K(x,\cdot)$ is concave in $B_\rho$ for every $x\in\R^d$;\smallskip
 \item[(b)]  the function $\max\{H(x,p),\mu(x)\}$ is convex in $p$ for every fixed $x\in\R^d$, with $\mu(x):=\inf_{|p|\geqslant\rho} H(x,p)$.\smallskip
\end{itemize}
Then there exists a family of functions $H_i\in\Bami$ such that $H$ can be written in the form \eqref{def H}. 

Let us prove the assertion. We first remark that $\mu$ is locally bounded on $\R^d$ since 
\[
-M_r\leqslant \mu(x)\leqslant \min_{|p|=\rho} H(x,p)\leqslant \Lambda(|\rho|^m+1)\quad\hbox{for all $x\in B_r$ and $r>0$,}
\]
in view of the fact that $H$ satisfies \eqref{h1.1}. 
% Moreover, still in view of \eqref{h1.1}, for every $r>0$ there exists a radius $\rho(r)\geqslant \rho$ such that \ $\ell(x)=\min_{\rho\leqslant |p|\leqslant R} H(x,p)$\ for every $x\in B_r$. As an infimum of a family of locally equi--Lipschitz functions, $\mu$ is itself locally Lipschitz.
This readily implies that the function \ $H_\flat(x,p):=\max\{H(x,p),\mu(x)\}$\ belongs to $\Bam$, for possibly different constants $\Lambda$, $(a_r)_{r>0}$, $(M_r)_{r>0}$. Note that $H(x,p)\geqslant \mu(x)$ for every $x\in\R^d$ and $|p|\geqslant\rho$, by definition of $\mu(x)$, in particular
\begin{equation}\label{relation Hflat}
H_\flat(x,p)=H(x,p)\qquad\hbox{for $x\in\R^d$ and $|p|\geqslant\rho$.} 
\end{equation}
By assumption, the function $F(x,p):=H(x,p)-K(x,p)$ is concave in $B_\rho$ with respect to $p$, for every fixed $x\in\R^d$, and Borel--measurable with respect to $(x,p)$. By well known result of convex analysis, for every fixed $q\in B_\rho$ we know that  
\begin{equation}\label{def superdifferential}
F(x,p)\leqslant \langle\xi,p-q\rangle +F(x,q)\qquad\hbox{for all $p\in B_\rho$,} 
\end{equation}
with equality holding at $p=q$, where $\xi$ is any vector in the superdifferential $\partial^+_p F(x,q)$, in the sense of convex analysis, of $F(x,\cdot)$ at $q$. By the measurable selection Theorem, see \cite[Theorem III.30]{Val}, we infer that there exists a Borel--measurable map $\xi:\R^d\times B_\rho\to\R^d$ such that $\xi(x,q)\in \partial^+_p F(x,q)$ for every $(x,q)\in \R^d\times  B_\rho$. Since the function $F$ satisfies condition (H3), it is easily seen that there exists a constant $C>0$ such that \ $|\xi(x,q)|\leqslant C$ for every $(x,q)\in\R^d\times B_\rho$.  
We introduce the set of indexes $\I:=B_\rho$ and for every $q\in\I$  we set 
\begin{align*}
G_{q}(x,p):=K(x,p)+\langle \xi(x,q),p-q\rangle +F(x,q)\qquad\hbox{for all  $(x,p)\in\R^d\times\R^d$}.
\end{align*}
Notice that, in view of \eqref{def superdifferential} and of the continuity of $H$ and $G_q$ in $p$, we have
\begin{equation}\label{property G_q}
H(x,p)\leqslant G_q(x,p)\qquad\hbox{for all $(x,p)\in \R^d\times\overline B_\rho$,}
\end{equation}
with equality holding at $p=q$. For every $q\in\I$ we set  
\begin{align*}
H_{q}(x,p):=
\begin{cases}
G_{q}(x,p) & \hbox{if $|p|\leqslant\rho$,}\\ 
\max\{ G_{q}(x,p), H_\flat(x,p)\} & \hbox{if $|p|>\rho$.}
\end{cases}
\end{align*}
Note that, for every fixed $x\in\R^d$, the function $H_q(x,\cdot)$ thus defined is continuous since $G_q(x,p)\geqslant H(x,p)=H_\flat(x,p)$  for every $|p|=\rho$, in view of \eqref{property G_q} and \eqref{relation Hflat}.
Furthermore, by construction, 
\begin{equation}\label{properties H_q}
H\leqslant H_q\quad\hbox{in $\R^d\times\R^d$\quad for each $q\in\I$,}\qquad
H=\inf_{q\in\I} H_q\quad\hbox{in $\R^d\times\overline B_\rho$,}
\end{equation}
An easy check shows that each $H_q$ satisfies \eqref{h1.1} and (H3) with $\ell$ in place of $m$ and for suitable constants $\tilde \Lambda$, $(\tilde a_r)_{r>0}$, $(\tilde M_r)_{r>0}$, independent of $q\in \I$. We claim that  $H_q(x,\cdot)$ is convex on $\R^d$, for every fixed $x\in\R^d$ and $q\in\I$. To prove this, we will show that the function $H_q(x,\cdot)$ possesses a subdifferential at each point $p_0\in\R^d$. If $p_0$ is such that $H_q(x,p_0)=G_q(x,p_0)$, it suffices to take a subdifferential of the convex function $G_q(x,\cdot)$ at $p_0$. Let us then assume $G_q(x,p_0)<H_q(x,p_0)=H_\flat(x,p_0)$, implying in particular that $|p_0|>\rho$. Let $\eta$ be a subdifferential of the convex function $H_\flat(x,\cdot)$ at $p_0$, i.e. 
\[
f(p):=H_\flat(x,p_0)+\langle \eta,p-p_0\rangle\leqslant H_\flat(x,p)\qquad\hbox{for every $p\in\R^d$.}
\]
To prove that $\eta$ is a subdifferential of $H_q(x,\cdot)$ on $\R^d$, it suffices to show that the function  $\varphi(p):=G_q(x,p)- f(p)$ is nonnegative on $\overline B_\rho$. But this is clearly true since $\varphi\geqslant 0$ on $\partial B_\rho$, $\varphi(p_0)<0$ and $\varphi$ is convex on $\R^d$. We conclude that $H_q\in\Bamltilde$ for every $q\in\I$.

\indent The asserted representation formula for $H$ is finally obtained by remarking that
\[
 H(x,p)=\inf_{q\in\I\cup\{\flat\}} H_q(x,p)\qquad\hbox{for all $(x,p)\in\R^d\times\R^d$}
\]
in view of \eqref{relation Hflat} and \eqref{properties H_q}. 
\end{esempio}

It would be very interesting, in Example \ref{ex 1}, to take as $G$ a concave function of $p$ of more general form, for instance such that 
$-G\in\Haml$ for some $\ell<m$; or to allow $\rho=+\infty$ in Example \ref{ex 2}, which is  basically an equivalent fact. Such an extension seems 
out of reach with the techniques we have employed. We remark that an analogous question was raised in \cite[Remark 2.1]{daLioLey08}. 

\medskip

\begin{appendix}
\section{}\label{appendix}

In this section we give a proof of Theorem \ref{teo parabolic comp} and Proposition \ref{prop Lip comp}.\\

\noindent{\em Proof of Theorem \ref{teo parabolic comp}.} 
We assume $\sup_{\partial_P\left(\TUcyl\right)}\big(v-w\big)<+\infty$, being the statement otherwise trivial. 
We set $\phi(x):=\sqrt{1+|x|^2}$ and remark that, due to hypothesis \eqref{app hyp 2}, the linear growth of $\phi$ at infinity and the upper semicontinuity of $v$ and $-w$, for every $\varrho>0$ there exists $\mu_\varrho>0$ such that 
\begin{equation}\label{parabolic ineq v}
v(t,x)\leqslant \varrho\phi(x)+\mu_\varrho,\quad -w(t,x)\leqslant \varrho\phi(x)+\mu_\varrho\qquad\hbox{for all $(t,x)\in\TUcyl$.}
\end{equation}
Fix $b>0$ and first observe that $\tilde w:=w+b/(T-t)$ satisfies 
\begin{equation}\label{parabolic supersol}
{\partial_t \tilde w}-\D{tr}(A(x)D_x^2 \tilde w)+H(x, D_x \tilde w)\geqslant \frac{b}{T^2}=: c\quad\hbox{in $\TUcyl$}. 
\end{equation}
Clearly, it is enough to prove the assertion for $v$ and $\tilde w$ for any fixed $b>0$. We will thus prove the comparison principle under the additional assumption that $w$ solves \eqref{parabolic supersol} and that, for every $\varrho>0$, there exists $\mu_\varrho>0$ such that 
\begin{equation}\label{parabolic ineq u}
-w(t,x)\leqslant \varrho\phi(x)+\mu_\varrho-\frac{b}{T-t}\qquad\hbox{for all $(t,x)\in\TUcyl$.}
\end{equation}
Moreover, up to adding to $v$ a suitable constant, we will also assume, without any loss of generality, that $\sup_{\partial_P\left(\TUcyl\right)}\big(v-w\big)=0$. The assertion is thus reduced to proving that $v\leqslant w$ in $\TUcyl$. We argue by contradiction: suppose that $v>w$ at some point of $(0,T)\times U $, which, up to translations, we can assume to be of the form  $(\overline t,0)$ for some $\overline t\in (0,T)$, and set $\theta:=v(\overline t, 0)-w(\overline t,0)>0$. 
Fix $\eta\in (0,\theta/4)$, $s\in (0,1)$ and $\eps\in (0,1)$, and consider the auxiliary function $\Phi:\cTUcyl\times \overline U \to \R$ defined by
\[
\Phi(t,x,y):=s v(t,x)-w(t,y)-\frac{|x-y|^2}{2\eps}-\eta\phi(x),\quad\hbox{$(t,x,y)\in\cTUcyl\times\overline U$.}
\]
Choose $s_0\in (1/2,1)$ sufficiently close to 1 so that 
\[
 \Phi(\overline t,0,0)=sv(\overline t,0)-w(\overline t,0)-\eta\phi(0)> \frac{\theta}{2}\quad\hbox{for all $\eta\in(0,\theta/4)$ and  $s\in (s_0,1)$.}
\]
By using \eqref{parabolic ineq v} and \eqref{parabolic ineq u}, a tedious but standard computation shows that there exists $(t_\eps,x_\eps,y_\eps)\in[0,T]\times\overline U \times\overline U $ such that 
\begin{equation}\label{parabolic raddoppio}
\Phi(t_\eps,x_\eps,y_\eps)=\sup_{\TUcyl\times U } \Phi \geqslant \Phi(\overline t,0,0)> \frac{\theta}{2}.
\end{equation}
By \cite[Lemma 3.1]{users}, up to subsequences,  
\begin{equation}\label{parabolic limits approximating points}
\lim_{\eps\to 0} (t_\eps,x_\eps,y_\eps)=(t_0,x_0,x_0)\quad\hbox{and}\quad\lim_{\eps\to 0}\frac{|x_\eps-y_\eps|^2}{\eps}=0
\end{equation}
for some $(t_0,x_0)\in\cTUcyl$ satisfying 
\begin{equation}\label{parabolic maximum point}
sv(t_0,x_0)-w(t_0,x_0)-\eta\phi(x_0)=\sup_{(t,x)\in\TUcyl} \Phi(t,x,x)> \frac{\theta}{2}.
\end{equation}
By exploiting inequalities \eqref{parabolic ineq v} and \eqref{parabolic ineq u}  with $\varrho:=\eta/4$ in \eqref{parabolic maximum point}, we easily get that any point $(t_0,x_0)\in\cTUcyl$ enjoying \eqref{parabolic maximum point} satisfies 
$$\eta\phi(x_0)+\frac{2b}{T-t_0}\leqslant 4\mu_{\eta/4}-\theta.$$
We infer that there exist a constant $R_\eta>1$, only depending on $\eta>0$, and a constant $T_{b,\eta}\in (0,T)$, depending on $b>0$ and $\eta>0$, such that 
$|x_0|\leqslant R_\eta-1$ and $t_0\leqslant T_{b,\eta}$. Furthermore, any such point $(t_0,x_0)$ actually lies in $\TUcyl$ provided
\begin{equation}\label{parabolic eq 1-s small}
(1-s)\leqslant \min\left\{4\eta, \frac{\theta}{2 \mu_{1/4}}\right\},
\end{equation}
where $\mu_{1/4}$ is the positive constant appearing in \eqref{parabolic ineq v} and \eqref{parabolic ineq u} with $\varrho=1/4$. Indeed, if $(t_0,x_0)\in \partial_P\left(\TUcyl\right)$, by exploiting the parabolic boundary condition $v\leqslant w$ on $\partial_P\left(\TUcyl\right)$, we get
\[
 \frac\theta 2<(1-s)\big(-w(t_0,x_0)\big)-\eta\phi(x_0)
 \leqslant
 \left(\frac{1-s}{4}-\eta\right)\phi(x_0)+(1-s)\mu_{1/4},
 \]
which is never satisfied as soon as $s$ is chosen as in \eqref{parabolic eq 1-s small}.

Let us hereafter choose $s=1-4\eta$ and $\displaystyle 0<\eta<\min\left\{1/8,{\theta}/{4},{\theta}/{(8\mu_{1/4})}\right\}$, so that $(t_0,x_0)\in\TUcyl$. In particular, $(t_\eps,x_\eps,y_\eps) \in \TUcyl\times\R^d$ for sufficiently small $\eps>0$. 
Now we use \eqref{parabolic raddoppio}, the fact that $v$ is a subsolution of \eqref{eq parabolic comp} and $w$ is a supersolution of \eqref{parabolic supersol}, and  \cite[Theorem 8.3]{users} to infer that there exist $\tau_\eps\in\R$ and symmetric $d\times d$ matrices $X_\eps,Y_\eps$ satisfying 
\begin{eqnarray*}
  -\frac{3}{\eps}
  \begin{pmatrix}
     I_d & 0 \\
     0 & I_d
  \end{pmatrix}
  \leqslant
  \begin{pmatrix}
     X_\eps & \,0 \\
     0 & -Y_\eps
  \end{pmatrix}
  \leqslant
  \frac{3}{\eps}
  \begin{pmatrix}
     \ I_d & \ -I_d \\
     -I_d & I_d
  \end{pmatrix}
\end{eqnarray*}
such that 
\begin{equation}\label{comp parabolic subsol}
\tau_\eps-\D{tr}\left(A(x_\eps)\Big(X_\eps+\eta D\phi(x_\eps)\Big)\right)+sH\left(x_\eps, \frac{p_\eps+q_\eps}{s}\right)\leqslant 0,
\end{equation}
\begin{equation}\label{comp parabolic supersol}
\tau_\eps-\D{tr}\left(A(y_\eps)Y_\eps\right)+H\left(y_\eps, p_\eps\right)\geqslant c,
\end{equation}
where we have set 
\[
p_\eps:=\frac{x_\eps-y_\eps}{\eps}\quad\hbox{and}\quad q_\eps:=\eta D\phi(x_\eps).
\]
As usual, the idea is to derive a contradiction by showing that the difference between \eqref{comp parabolic supersol} and \eqref{comp parabolic subsol} must be negative, after sending first $\eps\to 0^+$ and then $\eta\to 0^+$ (and consequently $s=1-4\eta\to 1^-$). 

To estimate difference between the terms involving $A$ in \eqref{comp parabolic supersol} and \eqref{comp parabolic subsol}, we argue as in the proof of Theorem 2.1 in \cite{AT} to get 
\begin{equation}\label{app estimate 3}
\D{tr}\left(A(x_\eps)\Big(X_\eps+\eta D^2\psi(x_\eps)\Big)-A(y_\eps)Y_\eps\right)\leqslant \tilde C\left(\frac{|x_\eps-y_\eps|^2}{\eps}+\eta\right)
\end{equation}
for some constant $\tilde C>0$ independent of $\eps$ and $\eta$. 
Therefore, by 
subtracting \eqref{comp parabolic subsol} from \eqref{comp parabolic supersol} and by taking into account \eqref{app estimate 3} and Lemma \ref{lemma AT}, we end up with
\begin{equation}\label{app estimate 4}
 0<c\leqslant \tilde C\left(\frac{|x_\eps-y_\eps|^2}{\eps}+\eta\right)+C(1-s)+C_\eta|x_\eps-y_\eps|.
\end{equation}
Now we send $\eps\to 0^+$ and then $\eta\to 0^+$ (and consequently $s=1-4\eta\to 1^-$) in \eqref{app estimate 4} and we obtain the sought contradiction, in view of \eqref{parabolic limits approximating points}.\medskip
\qed

We now proceed to give a proof of Proposition \ref{prop Lip comp}, that, as we will see, is derived via a minor modification from the one just presented. 
In what follows, we will denote by $D_x^+ v(t_\eps,x_\eps)$ the set of {\em superdifferentials} of the function $v(t_\eps,\cdot)$ at the point $x_\eps$, 
%i.e. the set of vectors $D_x\phi(x_\eps)$ with $\phi$ a supertangent  to $v(t_\eps,\cdot)$ at the point $x_\eps$. Analogously, we will denote by 
and by $D_x^- w(t_\eps,y_\eps)$ the set of {\em subdifferentials} of the function $w(t_\eps,\cdot)$ at the point $y_\eps$.\\ %namely the set defined by $D_x^- w(t_\eps,y_\eps):=-D_x^+(-w)(t_\eps,y_\eps)$.\\

\noindent{\bf Proof of Proposition \ref{prop Lip comp}.}
We argue as in the proof of Theorem \ref{teo parabolic comp} choosing now $s=1$. 
The only difference consists in the estimate of the term  
$
H\left(x_\eps, p_\eps+q_\eps \right)-H\left(y_\eps, p_\eps\right).
$
Notice that $p_\eps+q_\eps\in D^+_x v(t_\eps,x_\eps)$ and $p_\eps\in D^-_x w(t_\eps,y_\eps)$. From the fact that either 
$\|D_x v\|_{L^\infty(\TUcyl)}$ or $\|D_x w\|_{L^\infty(\TUcyl)}$ is finite, let us say less than a positive constant $\kappa$, and  that 
$|q_\eps|<\eta$, we infer  
\[
|p_\eps+q_\eps|\leqslant \kappa+\eta,\qquad |p_\eps|\leqslant \kappa+\eta.  
\]
Let us choose $\eta<1$ and let $\omega$ be a continuity modulus of $H$ in $U\times B_{\kappa+1}$. We have
\begin{equation}\label{app estimate 1}
|H\left(x_\eps, p_\eps+q_\eps \right)-H\left(y_\eps, p_\eps\right)|\leqslant \omega\left(|x_\eps-y_\eps|+\eta \right).
\end{equation}
The assertion follows by arguing as in the proof of Theorem \ref{teo parabolic comp} and by using \eqref{app estimate 1} in place of the inequality \eqref{claim AT} stated in Lemma \ref{lemma AT}.
\qed\\
\end{appendix}

\bibliography{viscousHJ}
\bibliographystyle{siam}

\end{document}